\documentclass[10pt,a4paper]{amsart}

\usepackage{hyperref}
\usepackage{color}

\usepackage[dvipdfm]{graphicx}
\usepackage[arrow]{xy}
\usepackage{pb-diagram,pb-xy}

\usepackage{amsmath,amsfonts,amssymb}

\theoremstyle{plain}
\newtheorem*{theorem*}{Theorem}
\newtheorem*{lemma*} {Lemma}
\newtheorem*{corollary*} {Corollary}
\newtheorem*{proposition*} {Proposition}
\newtheorem*{conjecture*}{Conjecture}

\newtheorem{theorem}{Theorem}[section]
\newtheorem{lemma}[theorem]{Lemma}
\newtheorem{corollary}[theorem]{Corollary}
\newtheorem{proposition}[theorem]{Proposition}

\theoremstyle{definition}

\theoremstyle{remark}
\newtheorem*{remark}{Remark}
\newtheorem*{definition}{Definition}

\newtheorem*{claim}{Claim}

\theoremstyle{definition}

\def\Q{\Bbb{Q}}

\def\Z{\Bbb{Z}}
\def\R{\Bbb{R}}
\def\C{\Bbb{C}}

\def\id{\operatorname{id}}
\def\gl{\operatorname{GL}}

\def\l{\lambda}
\def\ll{\langle}
\def\rr{\rangle}

\def\sm{\setminus}
\def\bp{\begin{pmatrix}}
\def\ep{\end{pmatrix}}
\def\bn{\begin{enumerate}}
\def\en{\end{enumerate}}

\def\Hom{\operatorname{Hom}}
\def\rank{\operatorname{rank}}

\def\ba{\begin{array}}
\def\ea{\end{array}}

\def\a{\alpha}

\def\s{\sigma}

\def\ti{\tilde}
\def\lk{\operatorname{lk}}
\def\fr12{\frac{1}{2}}
\def\diag{\operatorname{diag}}

\def\ol{\overline}

\def\Ker{\operatorname{Ker}}

\def\B{\mathcal{B}}
\def\BB{\mathcal{B}}
\def\CC{\mathcal{C}}
\def\DD{\mathcal{D}}
\def\EE{\mathcal{E}}
\def\VV{\mathcal{V}}
\def\FF{\mathcal{F}}

\def\zt{\Z[t^{\pm 1}]}

\def\End{\operatorname{End}}

\def\v{\varphi}

\def\to{\mathchoice{\longrightarrow}{\rightarrow}{\rightarrow}{\rightarrow}}
\makeatletter
\newcommand{\shortxra}[2][]{\ext@arrow 0359\rightarrowfill@{#1}{#2}}
\def\longrightarrowfill@{\arrowfill@\relbar\relbar\longrightarrow}
\newcommand{\longxra}[2][]{\ext@arrow 0359\longrightarrowfill@{#1}{#2}}
\renewcommand{\xrightarrow}[2][]{\mathchoice{\longxra[#1]{#2}}%
  {\shortxra[#1]{#2}}{\shortxra[#1]{#2}}{\shortxra[#1]{#2}}}
\makeatother

\begin{document}

\title{Twisted torsion invariants and link concordance}

%\date{\today}

\author{Jae Choon Cha}
\address{Pohang University of Science and Technology,
  Pohang Gyungbuk, Republic of Korea}
\email{jccha@postech.ac.kr}

\author{Stefan Friedl}
\address{University of Warwick, Coventry, UK}
\email{s.k.friedl@warwick.ac.uk}

% Hacking to amslatex: always MSC2000 (amslatex on arXiv server uses MSC1991)
% Jae Choon Cha <jccha@postech.ac.kr>
\def\subjclassname{\textup{2000} Mathematics Subject Classification}
\expandafter\let\csname subjclassname@1991\endcsname=\subjclassname
\expandafter\let\csname subjclassname@2000\endcsname=\subjclassname

\subjclass{57M25, 57M27, 57N70.}
\keywords{Twisted Torsion, Homology Cobordism, Link Concordance}

\begin{abstract}
  The twisted torsion of a 3-manifold is well-known to be zero
  whenever the corresponding twisted Alexander module is
  non-torsion. Under mild extra assumptions we introduce a new twisted
  torsion invariant which is always non-zero.  We show how this
  torsion invariant relates to the twisted intersection form of a
  bounding 4-manifold, generalizing a theorem of Milnor to the
  non-acyclic case. Using this result, we give new obstructions to
  3-manifolds being homology cobordant and to links being concordant.
  These obstructions are sufficiently strong to detect that the Bing
  double of the figure eight knot is not slice.
\end{abstract}

\dedicatory{Dedicated to the memory of Des Sheiham}

\maketitle

\section{Introduction}

In this paper we introduce new obstructions to links being concordant,
which are obtained from twisted torsion invariants. Recall that an
\emph{$m$--component (oriented) link} is an embedded ordered
collection of $m$ disjoint (oriented) circles in~$S^{3}$.  Given an
oriented link $L$ we denote by $-L$ the same link with the orientation
of each component reversed.  We say that two $m$--component oriented
links $L_0=L_0^1\cup \dots \cup L_0^m$ and $L_1=L_1^1\cup \dots \cup
L_1^m$ are \emph{concordant} if there exists a collection of $m$
disjoint, locally flat, oriented annuli $A_1,\dots,A_m$ in $S^{3}
\times [0,1]$ such that $\partial A_i= L_0^i\times 0 \cup -L_1^i\times
1, i=1,\dots,m$.  We say that an $m$--component link is \emph{slice}
if it is concordant to the trivial $m$--component link.  Equivalently
a link is slice if it bounds $m$ disjoint locally flat disks in
$D^{4}$.

Let $L\subset S^3$ be an oriented $m$--component link.  Throughout
this paper we write $E_L=S^3\sm \nu L$, where $\nu L$ is a tubular
neighborhood of $L$, and $\pi(L)=\pi_1(E_L)$.
Let $\psi\colon H_1(E_L)\to H$ be a non--trivial homomorphism to a
free abelian group and let $\a\colon \pi(L)\to \gl(R,k)$ be a
representation with $R$ a subring of $\C$.  We can then consider the
twisted homology module $H_1(E_L;R[H]^k)$ and we define
\[
\rank(L,\psi,\a)=\rank_{R[H]} H_1(E_L;R[H]^k),
\]
where $\rank_{R[H]} A$ denotes the dimension of $A\otimes_{R[H]}
  Q(H)$ over the quotient field $Q(H)$ of the integral domain~$R[H]$.
If $\rank(L,\psi,\a)=0$, then it turns out that $E_L$ is
$Q(H)^k$--acyclic, i.e., $H_*(E_L;Q(H)^k)=0$, and we can define the
torsion $\tau^{\a\otimes \psi}(L)\in Q(H)^\times$ which is
well--defined up to multiplication by an element of the form $\pm dh$
where $d\in \det(\a(\pi(L)))$ and $ h\in H$.  This invariant
generalizes the invariants introduced by Reidemeister, Milnor and
Turaev and later by Lin and Wada.

In general though $\rank(L,\psi,\a)$ will be non--zero, for example
this is the case for any boundary link with at least two components
(cf.\ Theorem \ref{thm:tauboundary}) and the trivial representation.
Suppose that $R\subset \C$ is closed under complex conjugation. In
that case the ring $R[H]$ has a natural involution given by complex
conjugation and $\ol{h}=h^{-1}$ for $h\in H$, this involution
furthermore extends naturally to an involution on~$Q(H)$.  Now suppose
that $\a\colon \pi(L)\to \gl(R,k)$ is a unitary representation and
suppose that $\psi$ is non--trivial on each meridian of $L$. Under
these assumptions we define a torsion invariant
\[
\tau^{\a\otimes \psi}(L)\in Q(H)^\times/N(Q(H))
\]
even if $\rank(L,\psi,\a)\ne 0$. Here $N(Q(H))$ denotes the subgroup
of norms of the multiplicative group $Q(H)^\times$, i.e., $N(Q(H))=\{
\pm q\ol{q} \, |\, q\in Q(H)^\times\}$. The torsion $\tau^{\a\otimes
  \psi}(L)$ viewed as an element in $Q(H)^\times/N(Q(H))$ is again
well--defined up to multiplication by an element of the form $\pm dh$
where $d\in \det(\a(\pi(L)))$ and $h\in H$.  The invariant
$\tau^{\a\otimes \psi}(L)$ is the twisted version of an invariant
first introduced by Turaev \cite[Section~5.1]{Tu86}.

\begin{remark}
  If $K$ is a knot in $S^3$, $\psi\colon H_1(E_K)\to \Z$ is an
  isomorphism, and $\a\colon \pi(L)\to \gl(\Z,1)$ is the trivial
  representation, then $\rank(K,\psi,\a)=0$ and $\tau^{\a\otimes
    \psi}(K)=\Delta_K(t)/(t-1)$ (cf.~\cite{Tu01}).  In general,
  if $\rank(L,\psi,\a)=0$, then $\tau^{\a\otimes \psi}(K)$ can be
  expressed in terms of twisted Alexander polynomials
  (cf.~\cite{KL99a} and \cite{FK06}). If $\a$ is the trivial
  representation and if $\rank(L,\psi,\a)>0$, then the torsion is
  related to the Alexander polynomial defined using
  $\mbox{Tor}_{\Z[H]}(H_1(E_L;\Z[H]))$
  (cf.~\cite[Theorem~5.1.1]{Tu86}). We expect that a similar result
  also holds in the twisted case.
\end{remark}

We can now state our sliceness obstruction. First, if $L$ is
the $m$--component unlink in $S^3$ with meridians $\mu_1,\dots,\mu_m$,
then given $\a$ and $\psi$ as above we will show in Corollary
\ref{cor:unlink} that $\rank(L,\psi,\a)=k(m-1)$ and
\[ \tau^{\a\otimes \psi}(L)=\pm dh \cdot \prod\limits_{i=1}^m
\det\big(\id-\psi(\mu_i)\a(\mu_i)\big)^{-1} \, \in
Q(H)^\times/N(Q(H))\] with $d\in \det(\a(\pi(L)))$ and $ h\in H$.
The following theorem says that the torsion invariant of a slice
  link is given by the above expression if the representation factors through
  a $p$-group.  This is a special case of our main result, Theorem
\ref{mainthm}, which is a more general result for link concordance.

\begin{theorem}\label{mainthmintro}
  Let $L$ be an $m$--component oriented slice link.  Let $R\subset \C$
  be a subring closed under complex conjugation and let $\a\colon
  \pi(L)\to \gl(R,k)$ be a representation which factors through a
  finite group of prime power order. Let $\psi\colon H_1(E_L)\to H$ be
  an epimorphism onto a free abelian group which is non--trivial on
  each meridian of $L$.  Then
  \[
  \rank(L,\psi,\a)=k(m-1)
  \]
  and
  \[
  \tau^{\a\otimes \psi}(L)=\pm dh\cdot \prod\limits_{i=1}^m
  \det\big(\id-\psi(\mu_i)\a(\mu_i)\big)^{-1} \, \in Q(H)^\times/
  N(Q(H))
  \]
  for some $d\in \det(\a(\pi(L)))$ and $h\in H$, where
  $\mu_1,\dots,\mu_m$ are meridians of~$L$.
\end{theorem}

\begin{remark}
  Let $K\subset S^3$ be a knot.  Fox and Milnor \cite{FM66} have shown
  that if $K$ is slice, then the untwisted Alexander polynomial of $K$
  factors as $\pm t^lf(t)f(t^{-1})$ for some polynomial $f(t)$ and
  some $l\in \Z$.  It is well--known that for a knot all
  representations which factor through a $p$--group are necessarily
  abelian, in particular one can easily show that the sliceness
  obstruction of Theorem \ref{mainthmintro} reduces to the Fox--Milnor
  obstruction. On the other hand for knots Kirk and Livingston
  \cite{KL99a} used twisted torsion corresponding to
  `Casson--Gordon'--type representations to give sliceness
  obstructions which go beyond Fox--Milnor.  We also refer to
  \cite{KL99b}, \cite{Ta02}, \cite{HKL08}, \cite{Liv09} and also
  \cite{FV09} for more on twisted Alexander polynomials of knots and
  their relation to knot concordance.
\end{remark}

\begin{remark}
\bn
\item If $\a$ is the trivial representation, then Theorem
  \ref{mainthmintro} was proved in the one--variable case by Murasugi
  \cite{Mu67} (cf.\ also \cite{Ka77}) and in the multi--variable case
  it was proved independently by Kawauchi \cite[Theorem~B]{Ka78} and
  Nakagawa \cite{Na78}. These results were reproved by Turaev
  \cite[Theorem~5.4.2]{Tu86} using torsion.
\item Theorem \ref{mainthmintro} is related in spirit to
  \cite[Theorem~2.2]{Fr05} where a sliceness obstruction for links is
  given using the Atiyah--Patodi--Singer eta invariant and unitary
  representations factoring through $p$--groups.  The metabelian case
  was considered earlier by the first author and Ko \cite{CKo99}.
\item We expect the obstruction of Theorem \ref{mainthmintro} to be
  closely related to the discriminant part of Hirzebruch-type
  invariants developed by the first author~\cite{Ch07}.
 \en
\end{remark}

We prove a generalization of Milnor's classical duality theorem for
Reidemeister torsion~\cite{Mi62} and use it as a key ingredient in the
proof of the above sliceness obstruction.  For an even dimensional
compact oriented manifold $W$ endowed with a homomorphism
$\Z[\pi_1(W)] \to Q$ into a field $Q$ with involution such that $W$ is
\emph{$Q$--acyclic}, Milnor proved that the Reidemeister torsion of
$M=\partial W$ over $Q$ is of the form $q\overline q$, where $q$ is
the Reidemeister torsion of~$W$.  In particular, up to norms, the
torsion of $M$ is trivial.  A generalization of this to the twisted
acyclic case is given in \cite{KL99a}.  We generalize these prior
results to the \emph{non-acyclic} case: \emph{the twisted torsion of
  an odd dimensional manifold $M$ is equal to the determinant of the
  twisted intersection pairing of a bounding even-dimensional manifold
  $W$ (up to norms)}.  For a precise statement, see
Theorem~\ref{thm:torsion-duality}.  We hope that this is of
independent interest and useful for further applications.

The paper is organized as follows.  In Section \ref{sec:twisted
  invariants} we introduce twisted torsion of 3-manifolds and links
with non-acyclic twisted homology and we show how these invariants
relate to intersection forms of bounding 4-manifolds.  In Section
\ref{section:cs} we study the torsion invariant of homology cobordant
3-manifolds and in Section \ref{section:lc} we apply these results to
link concordance where we prove our main theorem (which implies
Theorem \ref{mainthmintro}).  In Section \ref{section:prep} we discuss
and to a certain degree classify complex representations which factor
through $p$--groups.  In Section \ref{section:boundarylinks} we show
how to compute twisted torsion of boundary links using a boundary link
Seifert matrix and we discuss the behavior of our invariants for links
which are boundary slice.  Finally in Section
\ref{section:satellitebing} we prove a formula for the twisted
invariants of satellite links and we use this formula to reprove the
fact (first proved by the first author \cite{Ch07}) that the Bing
double of the Figure 8 knot is not slice.

\subsection*{Conventions}
 The following are understood unless it says specifically otherwise:
(i) groups are finitely generated,
(ii) manifolds are compact, orientable and connected,
(iii) rings are understood to be associative, commutative with unit element,
and
(iv) homology and cohomology groups are taken with integral coefficients.

\subsection*{Acknowledgments.}
We would like to thank Baskar Balasubramanyam, David Cimasoni, Taehee
Kim, Vladimir Turaev and Liam Watson for helpful comments and
conversations.  We also thank an anonymous referee for useful
suggestions.  The first author was supported by the National Research
Foundation of Korea (NRF) grants funded by Korean government(MEST)
(Grant No.\ 2009--0094069 and R01--2007--000--11687--0).

%===========================================
\section{Twisted torsion of 3-manifolds and  links}
\label{sec:twisted invariants}

In this section we define twisted torsion invariants of odd
dimensional manifolds for finite dimensional unitary representations
for which the corresponding twisted homology groups are not
necessarily acyclic.  We furthermore investigate twisted torsion of
odd dimensional manifolds cobounding a manifold, along the way we
generalize a theorem of Milnor to the nonacyclic case.  In the last
section we then specialize to the case of 3-manifolds and link
exteriors.

%===========================================
\subsection{Twisted homology and Poincar\'e duality}

Let $(X,Y)$ be a CW-pair with $\pi=\pi_1(X)$, and $\phi\colon \pi \to
\gl(k,R)$ a finite dimensional representation over a ring~$R$.  Denote
by $p\colon \widetilde{X}\to X$ the universal covering of $X$ and
write $\widetilde{Y}=p^{-1}(Y)$.  Recall that $\pi$ acts on the left
of $\widetilde{X}$ as the deck transformation group.  We consider the
cellular chain complex $C_*(\widetilde{X},\widetilde{Y})$ as a right
$\Z\pi$-module via $\s \cdot g:=g^{-1}\s$ for a chain~$\s$.  The
$R^k$--coefficient cellular chain complex of $(X,Y)$ is defined to be
\[
C_*(X,Y; R^k)=C_*(\widetilde{X},\widetilde{Y})\otimes_{\Z \pi} R^k.
\]
where $R^k$ is viewed as a $(\Z\pi,R)$--bimodule via the
representation~$\phi$.  We define \emph{the $i$th twisted homology
  group} to be the $R$-module
\[
H_i^\phi(X,Y;R^k) = H_i(C_*(X,Y;R^k)).
\]
If the representation is understood clearly, then we just write $H_i$
instead of~$H_i^\phi$.

 % and $H^i(X;S^k)=H^i_\a(X;S^k)$.

Now assume that $R$ is a ring with a (possibly trivial) involution
$\ol{\vphantom{a}\hphantom{a}}$.
Given an $R$--module $A$ we denote by $\ol{A}$ the opposite
$R$--module,
i.e. $A=\ol{A}$ as abelian groups, and the multiplication by $r\in R$
on $\ol{A}$ is given by the multiplication by $\ol{r}$ on $A$.
Finally recall that a representation $\phi\colon \pi \to \gl(k,R)$ is
called \emph{unitary} if $\phi(g^{-1})=\overline{\phi(g)}{}^t$ for all
$g\in \pi$.

We now have the following twisted Poincar\'e duality theorem (cf.\ also
\cite[p.\ 91]{Le94}, \cite[Lemma~4.12]{FK06}, and \cite{KL99a}). Here, given
a vector space $V$ over a field $Q$,
we denote the dual vector space
by $V^*=\Hom_Q(V,Q)$.

\begin{theorem}\label{thm:duality}
  Let $W$ be an $n$--manifold with $\partial W=M\cup M'$, $M$ and $M'$
  submanifolds such that $M\cap M'=\partial M=\partial M'$.  Let
  $\phi\colon \pi_1(W)\to \gl(k,Q)$ be a unitary representation over a
  field $Q$ with involution.  Then there exists a natural isomorphism
  \[
  H_i(W,M;Q^k)\cong \ol{H_{n-i}(W,M';Q^k)^*}
  \]
  of $Q$--vector spaces.
\end{theorem}

The theorem can be proved along the same lines as the standard proof
of Poincar\'e duality.  For the reader's convenience we give an
outline of the proof.

\begin{proof}
  Choose a cell structure of $(W,M)$ to define $C_*(\widetilde W
  ,\widetilde M)$ and let $D_*(\widetilde W, \widetilde M')$ be the
  chain complex of the dual cell structure.
  % Then we have the
  % equivariant intersection pairing
  % \[
  % \langle \hphantom{x} \mbox{,} \hphantom{x}\rangle\colon
  % C_*(\widetilde W, \widetilde M) \times D_{n-*}(\widetilde
  % W,\widetilde M') \to \Z\pi
  % \]
  % which is nonsingular and sesquilinear, i.e., $\langle xg, yh\rangle
  % = g^{-1}\langle x, y\rangle h$ for $g,h\in \pi$.  As mentioned
  % in~\cite[p.\ 91]{Le94}, this gives rise to a nonsingular sesquilinear
  % intersection pairing
  % \[
  % \langle \hphantom{x} \mbox{,} \hphantom{x}\rangle\colon C_*(W,M;Q^k)
  % \times D_{n-*}(W,M';Q^k) \to Q
  % \]
  % which is given by $\langle x\otimes u, y\otimes v \rangle = \ol{u}^t
  % \phi(\langle x,y \rangle)v$. (Here we need the assumption
  % that $\phi$ is unitary.)
  Then the equivariant intersection pairing (e.g., see ~\cite[p.\
  91]{Le94}) gives an isomorphism of chain complexes
  \[
  C_*(W,M;Q^k) \cong \ol{D_{n-*}(W,M';Q^k)^*}.
  \]
  Applying the universal coefficient theorem over $Q$, the conclusion
  follows.
\end{proof}

\subsection{Basic definitions of twisted torsion}\label{section:rt}

\subsubsection*{Torsion of a based chain complex}
We begin by recalling the algebraic setup for torsion invariants.
Suppose $C=\{C_*\}$ is a based chain complex over a field $Q$ and
$\B=\{\B_*\}$ is a basis for $H_*(C)$, i.e., $\B_i$ is a basis of the
$Q$--vector spaces $H_i(C)$.  The \emph{torsion} $\tau(C, \B) \in
Q^\times := Q\sm \{0\}$ is defined as in \cite{Tu01,Tu02}.  (See also
Milnor's classic introduction to torsion \cite{Mi66}, but note that we
follow Turaev's convention, which gives the reciprocal of the torsion
in \cite{Mi66}.) If $H_*(C)$ is identically zero, then we will just
write $\tau(C)\in \Q^\times$ for the torsion.

In the following theorem, we collect well-known algebraic properties
of torsion which will be useful later.  In the theorem, given two
bases $b$ and $b'$ of a vector space, we will denote the basis change
matrix from $b$ to $b'$ by $(b|b')$, as in Milnor~\cite{Mi66} and
Turaev~\cite{Tu01,Tu02}.

\begin{theorem}\label{thm:alg-prop-torsion}\label{thm:mi66}
  Suppose $C$ is a based chain complex over $Q$ and $\B=\{B_i\}$ is a
  basis of $H_*(C)$.
  \begin{enumerate}
  \item Suppose $\B'=\{\B_i'\}$ is another basis of $H_*(C)$.  Then
    \[
    \tau(C,\B)=\tau(C,\B') \cdot \prod_i (\B_i | \B_i')^{(-1)^{i+1}}.
    \]
  \item Let $C'$ be the dual based chain complex given by $C'_i
    =(C_{n-i})^*$ and $\B'$ be the basis of $H_*(C') =
    H_{n-*}(C)^*$ dual to~$\B$.  Then $\tau(C,\B) =
    \tau(C',\B')^{(-1)^{n+1}}$.
  \item Suppose $C_*$ is acyclic.  Choose a basis of the $n$th cycle
    submodule $Z_n = \Ker\{C_n \to C_{n-1}\}$, and
    view
    \[
    \begin{aligned}
      C' &= \{\, Z_n \to C_n \to \cdots\to C_0 \,\}\\
      C'' &= \{\, \cdots \to C_{n+2} \to C_{n+1} \to Z_n \,\}
    \end{aligned}
    \]
    as acyclic based chain complexes (indexed so that $C'_0 = C_0$,
    $C''_0 = Z_n$).  Then $\tau(C) = \tau(C')\cdot
    \tau(C'')^{(-1)^n}$.
  \item Suppose $0\to C'\to C\to C''\to 0$ is a short exact sequence
    of based chain complexes and $\B'$, $\B$, and $\B''$ are bases of
    $H_*(C')$, $H_*(C)$ and $H_*(C'')$, respectively.  We view the
    associated homology long exact sequence as an acyclic complex, say
    $H$, based by $\B$, $\B'$, $\B''$.  Then
    \[
    \tau(C,\B)=\tau(C',\B')\cdot \tau(C'',\B'')\cdot \tau(H).
    \]
  \end{enumerate}
\end{theorem}

% Since the proofs of the above statements are well known (or at least
% follow from standard arguments), we omit details.

\subsubsection*{Twisted torsion of CW-complexes}
Let $(X,Y)$ be a finite CW-pair with $\pi=\pi_1(X)$ and let $\phi
\colon \pi\to \gl(k,Q)$ be a representation over a field~$Q$.  Let
$\mathcal{B}=\{\B_*\}$ be a basis of $H_*(X,Y;Q^k)$.  The universal
cover $(\tilde X, \tilde Y)$ has a natural cell structure, and the
chain complex $C_*(\widetilde X, \widetilde Y)$ can be based over
$\Z\pi$ by choosing a lift of each cell of $(X,Y)$ and orienting it.
This, together with the standard basis of $Q^k$, gives rise to a
basing of $C_*(X,Y;Q^k)$ over~$Q$.  We can then define the
\emph{twisted torsion}
\[
\tau^\phi(X,Y,\mathcal{B})\in Q^\times
\]
to be the torsion of $C_*(X,Y;Q^k)$ with respect to~$\B$.  We will
drop $\B$ from the notation if $H_*(X,Y;Q^k)=0$.

Standard arguments (e.g., see \cite{Mi66, Tu86, Tu01, Tu02}) show that
$\tau^\phi(X,Y,\mathcal{B})$ is well-defined up to multiplication by
an element in $\pm \det(\phi(\pi))$, and is invariant under simple
homotopy preserving~$\B$.  By Chapman's theorem \cite{Chp74} the
invariant $\tau^\phi(X,Y,\mathcal{B})$ only depends on the
homeomorphism type of $(X,Y)$.  In particular when $(M,N)$ is a
manifold pair, we can define $\tau^\phi(M,N,\mathcal{B})$ by picking
any finite CW--structure for $(M,N)$.

\subsubsection*{Twisted torsion of odd dimensional manifolds}
\label{section:twitor}

Let $M$ be an odd dimensional manifold with $\pi=\pi_1(M)$ and let
$\phi\colon \pi \to \gl(k,Q)$ be a representation over a field $Q$.
In many interesting cases $H_*(M;Q^k)$ will be nontrivial.  For
example the untwisted multivariable Alexander module of a boundary
link with at least two components is always non-torsion.  In order to
define an invariant for the nonacyclic case without referring to a
basis of homology, we now assume the following two conditions hold:

\begin{enumerate}
\item $Q$ is endowed with a (possibly trivial) involution, and $\phi$
  is unitary.
\item $\partial M$ is $Q^k$-acyclic, i.e., $H_*(\partial M;Q^k)=0$.
\end{enumerate}

Since $\phi$ is unitary, we have the Poincar\'e duality isomorphism
\[
H_i(M;Q^k)\cong \ol{H_{n-i}(M,\partial M;Q^k)^*}\cong \ol{H_{n-i}(M;Q^k)^*}
\]
of $Q$--vector spaces, by Theorem \ref{thm:duality} and by
Condition~(2).  Since $M$ is odd dimensional, one can pick a basis
$\B=\{\B_*\}$ for $H_*(M;Q^k)$ with the following property: for
each~$i$, $\B_{i}$ is the dual basis of $\B_{n-i}$ via the above
Poincar\'e duality isomorphism.  We call such a basis $\B=\{B_*\}$ a
\emph{self-dual basis} for $H_*(M;Q^k)$.

\begin{lemma}\label{lem:changebasis}
  Suppose $\B$ and $\B'$ are self-dual bases for $H_*(M;Q^k)$.  Then
  for some $q\in Q^\times$, up to the indeterminacy of the torsion,
  \[
  \tau^{\phi}(M,\B')=\tau^{\phi}(M,\B) \cdot q\ol{q}.
  \]
\end{lemma}

\begin{proof}
  Let $(\B_i|\B'_i)$ be the determinant of the base change matrix from
  $\B_i$ to $\B'_i$.  Then $(\B_{n-i}|\B'_{n-i}) =
  \overline{(\B_i|\B'_i)}{\mathstrut}^{-1}$.  The desired conclusion
  follows immediately from Theorem~\ref{thm:alg-prop-torsion}.
\end{proof}

We define the \emph{norm subgroup} of $Q^\times$ to be $N(Q)=\{\pm
q\ol{q} \mid q\in Q^\times\}$, and we say $f\in Q^\times$ is a
\emph{norm} when $f\in N(Q)$.  We define
\[
\tau^{\phi}(M) = \tau^{\phi}(M,\B) \text{ as an element in }
Q^\times/ N(Q)
\]
where $\B$ is a self-dual basis of $H_*(M;Q^k)$.  By
Lemma~\ref{lem:changebasis} the invariant $\tau^\phi(M)$ is
well-defined up to multiplication by an element in $\pm
\det(\phi(\pi_1(M)))$.  This invariant was first introduced by Turaev
\cite[Section~5.1]{Tu86} in the untwisted case.

In the following, given $f\in Q$, we will sometimes write
$\tau^{\phi}(M)\doteq f \in Q^\times/N(Q)$, to indicate that there
exists a representative of $\tau^{\phi}(M)$ which equals $f$.

\subsection{Twisted torsion and bounding manifolds}

In this subsection we prove a \emph{non-acyclic} generalization
of a well-known theorem of Milnor~\cite[Theorem~2]{Mi62} and of its
twisted analogue due to Kirk and Livingston~\cite[Theorem 5.1 and
Corollary 5.3]{KL99a}.

Let $W$ be a $2r$--dimensional manifold with (possibly disconnected)
boundary~$M$.  Let $\phi\colon\pi_1(W) \to \gl(k,Q)$ be a unitary
representation.  Now consider the map
\[
H_r(W;Q^k)\xrightarrow{\cong}
\ol{H_{r}(W,M';Q^k)^*}\xrightarrow{\iota^*} \ol{H_{r}(W;Q^k)^*}
\]
where the first map is the isomorphism given by Theorem
\ref{thm:duality} and the second map is induced by $\iota\colon
(W,\emptyset) \to (W,M)$.  This map gives rise to a pairing
\[
\l\colon H_r(W;Q^k)\to H_r(W;Q^k)\to Q
\]
which is well-known to be $(-1)^r$-hermitian. This pairing is called
the \emph{equivariant intersection form of $(W,\phi)$}.  This form is
in general singular, in fact for any $x\in \mbox{Im}\{H_r(M;Q^k)\to
H_r(W;Q^k)\}$ and $y\in H_r(W;Q^k)$ we have $\l(x,y)=\l(y,x)=0$.  In
particular $\l$ gives rise to a pairing on $H_r(W;Q^k)/i_* H_r(M;Q^k)$
which turns out to be non-singular.

We now pick a basis $\BB=\{v_1,\dots,v_s\}$ for $H_r(W;Q^k)/i_*
H_r(M;Q^k)$ and we compute $\det(\l(v_i,v_j))\in Q^\times$. Note that
if we change the basis, then the determinant of the form changes by a
norm.  Put differently, we obtain a well-defined invariant
\[
\Lambda(W,\phi):=\det(\l(v_i,v_j))\in Q^\times/N(Q).
\]
Note that $\Lambda(W,\phi)=\ol{\Lambda(W,\phi)}$ since $\l$ is
$(-1)^r$-hermitian.

\begin{theorem}\label{thm:torsion-duality}
  Suppose $W$ is a $2r$--dimensional manifold with (possibly
  disconnected) boundary~$M$.  Let $\phi\colon\pi_1(W) \to \gl(k,Q)$
  be a unitary representation.  Then
  \[ \tau^\phi(M) = \Lambda(W,\phi)\in Q^\times/N(Q)\] up to the
  indeterminacy $\pm \det(\phi(\pi_1(M)))$.
\end{theorem}

In the following we will only apply the theorem to the case that
$\Lambda(W,\phi)=1$, but we hope that the general case is of
independent interest.
% In particular the general case is needed to study the precise
% relationship between the obstructions of this paper and the
% invariants introduced in~\cite{Ch07}.

To prove Theorem~\ref{thm:torsion-duality} we need the following
duality of torsion, which is essentially due to Milnor~\cite{Mi62}.
(See also Kirk and Livingston~\cite{KL99a} for the twisted case.)

\begin{lemma}\label{lem:torsion-duality}
  Suppose $W$ is an $n$--manifold and $M$, $M'$ are submanifolds of
  $\partial W$ such that $\partial W=M\cup M'$ and $M \cap M'=\partial
  M=\partial M'$.  Suppose $\phi\colon \pi_1(W) \to \gl(k,Q)$ is a
  unitary representation, $\B=\{\B_*\}$ is a basis for $H_*(W,M;Q^k)$,
  and $\B'$ is the dual basis for
  $H_*(W,M';Q^k)=\overline{H_{n-*}(W,M;Q^k)}{\mathstrut}^*$.  Then
  \[
  \tau^\phi(W,M,\B) =
  \overline{\tau^\phi(W,M',\B')}{\mathstrut}^{(-1)^{n+1}}.
  \]
\end{lemma}

\begin{proof}
  The lemma follows immediately from the duality
  \[ C_i(W,M;Q^k)\cong
  \overline{D_{n-i}(W, M';Q^k)}{\mathstrut}^*\] (see the proof of
  Theorem~\ref{thm:duality}) and Theorem~\ref{thm:alg-prop-torsion}.
\end{proof}

\begin{proof}[Proof of Theorem~\ref{thm:torsion-duality}]
  Choose a basis $\B$ for $H_*(W;Q^k)$, and choose a self-dual basis
  $\B'$ for $H_*(M;Q^k)$.  Let $\B''$ be the basis of
  $H_*(W,M;Q^k)=\overline{H_{n-*}(W;Q^k)}{\mathstrut}^*$ which is dual
  to~$\B$.  From (the proof of) Theorem \ref{thm:duality},
  Theorem~\ref{thm:alg-prop-torsion}, and
  Lemma~\ref{lem:torsion-duality}, it follows that
  \[
  \tau^\phi(W,\B)= \tau_0 \cdot
  \tau^\phi(M,\B') \cdot \overline{\tau^\phi(W,\B)}^{(-1)^{n+1}}
  \]
  where $\tau_0$ is the torsion of the acyclic chain complex ($=$
  homology long exact sequence)
  \[
  \cdots \to H_{i+1}(W,M;Q^k) \to H_i(M;Q^k) \to H_i(W;Q^k) \to
  H_i(W,M;Q^k) \to \cdots
  \]
  which is based by $\B$, $\B'$, $\B''$.  We break the long exact
  sequence at $i_*\colon H_r(W;Q^k) \to H_r(W,M;Q^k)$ as in
  Theorem~\ref{thm:alg-prop-torsion}: let $P$ be the image of $i_*$,
  choose a basis $b$ for $P$, and let
  \[
  \begin{aligned}
    C' &= \{ \, P \to H_r(W,M;Q^k) \to \cdots \to
    H_0(W,M;Q^k) \, \}, \\
    C'' &= \{ \, \cdots \to H_r(M;Q^k) \to H_r(W;Q^k) \to P \, \}.
  \end{aligned}
  \]
  Then by Theorem~\ref{thm:alg-prop-torsion} (3) we have $\tau_0 =
  \tau(C') \tau(C'')^{(-1)^r}$.  Since $\B'$ is self-dual,
  $C''=\{C''_i\}$ is canonically isomorphic, as a based chain complex, with the
  dual chain complex $\{\overline{C'_{3r+1-i}}{\mathstrut}^*\}$ of
  $C'$ except $C''_0 = P$.  Therefore we have
  \[
  \tau(C'')=\overline{\tau(C')}^{(-1)^r} \cdot \big(\overline{b^*}|b
  \big)
  \]
  where $b^*$ is the dual basis of $b$ for $P^*$.  It follows that
  \[
  \tau^\phi(M,\B')= \tau^\phi(W,\B) \cdot \overline{\tau^\phi(W,\B)}
  \cdot \tau(C')^{-1} \cdot \overline{\tau(C')}{}^{-1} \cdot
  \big(b|\overline{b^*}\big)^{(-1)^r}.
  \]
  Note that by definition we have $\big(b|\overline{b^*})=\Lambda(W,\phi)=:D\in Q^\times/N(Q)$.
  Since the intersection form is $(-1)^r$-hermitian, $D=\pm\overline
  D$ and so $D^{-1}\equiv D$ in $Q^\times/N(Q)$.  This
  completes the proof.
\end{proof}

\begin{remark}
  If $W$ is $Q^k$--acyclic, then $D=1$ automatically and $\tau(C')=1$
  in the above proof.  So $\tau^\phi(M)=q\overline{q}$ up to $\pm\det
  (\phi(\pi_1(M)))$, where $q=\tau^\phi(W)$.  This shows our
  non-acyclic result specializes to Milnor's
  theorem~\cite[Theorem~2]{Mi62} and its twisted analogue due to Kirk
  and Livingston~\cite[Theorem 5.1 and Corollary 5.3]{KL99a}.
\end{remark}

\subsection{Twisted torsion of 3-manifolds and links}\label{section:twi3mfd}

In this section we now specialize to the case of 3-manifolds.  Let $M$
be a 3-manifold with empty or toroidal boundary (e.g. a link
exterior). We write $\pi=\pi_1(M)$.  Let $\psi\colon \pi\to H$ be an
epimorphism onto a nontrivial free abelian group~$H$.  Let $R$ be a
domain with involution. We equip $R[H]$ with the involution given by
$\ol{rh}=\ol{r}h^{-1}$ for $r\in R$ and $h\in H$. This extends to an
involution on $Q(H)$, the quotient field of $R[H]$.  Now let
$\alpha\colon\pi\to \gl(k,R)$ be a unitary representation.  Using $\a$
and $\psi$, we define a left $\Z[\pi]$-module structure on $R[H]^k :=
R^k\otimes_R R[H]$ as follows:
\[
g\cdot (v\otimes p):= (\a(g)\cdot v)\otimes (\psi(g)p)
\]
where $g\in \pi$, $v \in R^k$ and $p\in R[H]$.  This extends to a
$\Z[\pi]$-module structure on $Q(H)^k$.  It can be seen that
$\alpha\otimes \psi\colon \pi \to \gl(k,Q(H))$ is unitary since
$\alpha$ is unitary.  We will several times make use of the following
simple fact:
\[
\det((\a\otimes \psi)(\pi))\subset \det(\a(\pi))\cdot H.
\]
We now say that $\psi\colon \pi\to H$ is \emph{admissible} if $\psi$
restricted to any boundary component of $M$ is nontrivial. Note that
$\psi$ is always admissible if $M$ is closed.  If $\psi$ is admissible
then Condition (2) in Section \ref{section:twitor} is satisfied (e.g.,
see \cite[Proposition~3.5]{KL99a} and \cite[Section~3.3]{KL99a}) and
by the discussion of the previous section we thus obtain an invariant
\[
\tau^{\a\otimes \psi}(M)\in Q(H)^\times/N(Q(H))
\]
well-defined up to multiplication by an element of the form $\pm dh$
with $d\in \det(\a(\pi))$ and $h\in H$.

\begin{remark}
  Note that if $\rank(H)$ is nonzero, then in the above setting we
  have $H_i(M;Q(H)^k)=0$ for $i=0,3$ (see e.g. \cite{FK06,FK08}), so
  that it suffices to choose $\B_1$ in the definition of torsion.
\end{remark}

We now specialize even further, namely to the case of link exteriors.
Let $L\subset S^3$ be an oriented $m$--component link.  We denote the
exterior by $E_L$ and the (oriented) meridians by $\mu_1,\dots,\mu_m$.
Using the basis $\mu_1,\dots,\mu_m$ we can now naturally identify
$H_1(E_L;\Z)$ with $\Z^m$.  We say that a homomorphism
$\psi\colon\Z^m\to H$ is \emph{admissible} if $\psi$ is an epimorphism
onto a nontrivial free abelian group $H$ such that the epimorphism is
nontrivial on each subsummand of $\Z^m=\Z \oplus \dots \oplus \Z$.  In
this case the induced map $\pi_1(E_L) \to H$ is admissible in the
above sense, so that we obtain a well-defined torsion invariant
\[
\tau^{\a\otimes \psi}(L) :=\tau^{\a\otimes \psi}(E_L) \in
Q(H)^\times / N(Q(H)).
\]

%==============================================================
\section{Homology cobordism and twisted torsion}\label{section:cs}

In this section we investigate the behaviour of the twisted torsion
under homology cobordism.  Recall that two $3$--manifolds $M$ and $M'$
(possibly with nonempty toroidal boundary) are said to be \emph{homology
  cobordant} if there exists a $4$--manifold $W$ containing $M$ and $M'$
as submanifolds such that $\partial W = M \cup -M'$, $\partial M = M
\cap M' = \partial M'$, and the inclusions $M \to W$ and $M' \to W$
induce isomorphisms on $H_*(-;\Z)$.

From now on, $R\subset \C$ is always assumed to be a subring closed
under complex conjugation. We denote by $P_k(R,\pi)$ the set of
representations $\pi\to \gl(k,R)$ factoring through a $p$-group. We
will continuously make use of the well-known fact that complex
representations factoring through finite groups are necessarily
unitary.

For $\psi\colon \pi_1(M) \to H$ and $\alpha\in
P_k(R,\pi)$, we define
\[ \rank(M,\alpha,\psi):=\dim_{Q(H)}
H_1^{\alpha\otimes\psi}(M;Q(H)^k).\]  Our main result is the following:

\begin{theorem}\label{thm:homology-cob-inv}
  Suppose two $3$--manifolds $M$ and $M'$ are homology cobordant and
  $H$ is a free abelian group.  Then there are bijections
  \begin{gather*}
    \Psi\colon
    \Hom(\pi_1(M),H) \to \Hom(\pi_1(M'),H) \\
    \Phi\colon P_k(R,\pi_1(M)) \to P_k(R,\pi_1(M'))
  \end{gather*}
  such that for any $\psi\colon \pi_1(M) \to H$, $\psi$ is admissible
  if and only if $\Psi(\psi)$ is admissible, and in this case, for any
  $\alpha\in P_k(R,\pi_1(M))$, we have
  \[
  \rank(M,\alpha,\psi)=\rank(M',\Phi(\alpha),\Psi(\psi))
  \]
  and
  \[
  \tau^{\alpha\otimes\psi}(M) =\pm dh
  \tau^{\Phi(\alpha)\otimes\Psi(\psi)}(M') \text{ in } Q(H)^\times
  / N(Q(H))
  \]
  for some $d\in \det(\a(\pi_1(M)))=\det((\Psi(\alpha))(\pi_1(M')))$ and $h\in H$.
\end{theorem}

\begin{remark}
  In the proof of Theorem \ref{thm:homology-cob-inv} we will see that
  for a given homology cobordism $W$ between $M$ and $M'$, the
  maps $\Psi$ and $\Phi$ are both induced by $\pi_1(M)\to
  \pi_1(W)\leftarrow \pi_1(M')$, in particular $\Psi$ and $\Phi$ are
  `related' bijections.
\end{remark}

\subsection{Stallings Theorem and representations through $p$-groups}
\label{subsec:correspondence-of-representations}

We first construct the bijections in
Theorem~\ref{thm:homology-cob-inv} using Stallings'
theorem~\cite{St65}.

\begin{proof}[Proof of Theorem~\ref{thm:homology-cob-inv}, Part I]
  Suppose $W$ is a homology cobordism between $M$ and~$M'$.  Obviously
  the induced isomorphisms on $H_1(-;\Z)$ gives rise to bijections
  \[
  \Hom(\pi_1(M),H) \approx \Hom(\pi_1(W),H) \approx \Hom(\pi_1(M'),H).
  \]
  Their composition will be $\Psi$ in the above statement.  Since
  $\partial M = \partial M'$ it follows that $\psi\colon \pi_1(M) \to H$ is
  admissible if and only if so is $\Psi(\psi)$.

  The bijection $\Phi$ is constructed as follows.  For a group $G$, we
  denote the $q$th lower central subgroup by~$G_q$, which is defined
  inductively via $G_1=G$, $G_{q}=[G,G_{q-1}]$.  For any $\alpha\in
  P_k(R,\pi_1(M))$, $\alpha$ factors through $\pi_1(M)/\pi_1(M)_q$ for
  some $q$, since any $p$-group is nilpotent.  By Stallings'
  theorem~\cite{St65}, we have
  \[
  \pi_1(M)/\pi_1(M)_q \xrightarrow{\cong} \pi_1(W)/\pi_1(W)_q.
  \]
  Therefore $\alpha$ induces a representation $\pi_1(W) \to \gl(k,R)$
  which factors through $\pi_1(W)/\pi_1(W)_q$.  It is easily seen that
  this induces a bijection $P_k(R,\pi_1(M)) \xleftarrow{\approx}
  P_k(R,\pi_1(W))$.  Similarly for $M'$ and $W$, and composing them,
  we obtain a bijection $\Phi\colon P_k(R,\pi_1(M)) \to
  P_k(R,\pi_1(M'))$.
  \def\qedsymbol{}
\end{proof}

\subsection{Cohn local property and twisted coefficients}

In our proof of the conclusion on torsions in
Theorem~\ref{thm:homology-cob-inv} we will compute (the quotient of)
the torsions of $M$ and $M'$ in terms of the $Q(H)^k$-coefficient
intersection pairing of a cobordism $W$, appealing to
Theorem~\ref{thm:torsion-duality}.  The key point is that one can
prove that the intersection pairing is trivial when $W$ is a homology
cobordism.  This is done following a standard strategy, namely by
controlling the size of the underlying $Q(H)^k$-coefficient homology
of $(W,M)$, appealing to a chain contraction argument originally due
to Vogel.

In Lemma~\ref{lem:cohn-local-prop} below we denote by $\epsilon\colon
\Z[\pi] \to \Z$ the augmentation map given by $g\mapsto 1, g\in \pi$.

\begin{lemma}\label{lem:cohn-local-prop} \label{thm:levine}
  Suppose $\pi$ is a group and $R$ is a domain with characteristic
  zero, $\alpha\in P_k(R,\pi)$, and $\psi\colon \pi \to H$ is an
  epimorphism onto a nontrivial free abelian group~$H$.  If $A$ is a
  square matrix over $\Z[\pi]$ such that $\epsilon(A)$ is invertible
  over $\Z$, then $(\alpha\otimes\psi)(A)$ is invertible over $Q(H)$.
\end{lemma}

\begin{remark}
  Another way of phrasing Lemma~\ref{lem:cohn-local-prop} is to say
  that the $\Z[\pi]$-module $Q(H)^k$ is in fact a module over the
  \emph{Cohn localization} of~$\Z[\pi]$.
\end{remark}

Our proof of Lemma~\ref{lem:cohn-local-prop} depends heavily on a
result originally due to Strebel \cite{Str74} and Levine~\cite{Le94}.

\begin{proof}
  Since $R$ has characteristic zero and since the quotient ring of $R$
  gives the same $Q(H)$, we may assume that $R$ contains $\Q$.
  Suppose $\alpha$ factors through $f\colon \pi \to P$ where $P$ is a
  $p$-group.  We have the following commutative diagram:
  \[
  \begin{diagram}
    \node{\End_{\Z[\pi]}(\Z[\pi]^n)} \arrow{e,t}{\alpha\otimes\psi}
    \arrow{s,l}{f} \arrow{se,t}{\alpha}
    \node{\End_R(R[H]^{nk})} \arrow{e,t}{\det} \arrow{s,r}{\epsilon_R}
    \node{R[H]} \arrow{s,r}{\epsilon_R}
    \\
    \node{\End_{\Q[P]}(\Q[P]^n)} \arrow{e}
    \node{\End_R(R^{nk})} \arrow{e,t}{\det}
    \node{R}
  \end{diagram}
  \]
  where the induced maps are denoted by the same symbol, as an abuse of
  notation.  The augmentation $R[H] \to R$ is denoted by~$\epsilon_R$.

  Suppose $A\in \End_{\Z[\pi]}(\Z[\pi]^n)$ is an $n\times n$ matrix
  over $\Z[\pi]$ such that $\epsilon(A)$ is invertible.  Due to
  Strebel~\cite[Lemma~1.10]{Str74} and Levine \cite[Lemma~4.3]{Le94} this implies that  $f(A)$ is invertible
  over $\Q[P]$, i.e., a unit in $\End_{\Q[P]}(\Q[P]^n)$.  Since the
  maps in the above diagram are multiplicative, it follows that
  $\alpha(A)$ is invertible.  Therefore $\det(\alpha(A))\ne 0$ and
  $\det((\alpha\otimes\psi)(A))\ne 0$.  It follows that $A$ is
  injective over $R[H]$ and hence invertible over the field $Q(H)$.
\end{proof}

A standard well-known chain contraction argument originally due to
Vogel (see also Levine~\cite[Proposition~4.2]{Le94} and
Cochran-Orr-Teichner~\cite{COT03}) shows the following statement as a
consequence of Lemma~\ref{lem:cohn-local-prop}:

\begin{lemma}\label{lem:homology-cohn-local-coeff}
  Suppose $(X,A)$ is a finite CW-pair, $R$ is of characteristic zero,
  $\alpha\in P_k(R,\pi_1(X))$, and $\psi\colon \pi_1(X) \to H$ is an
  epimorphism onto a nontrivial free abelian group.  If
  $H_*(X,A;\Z)=0$, then $H_*(X,A;Q(H)^k)=0$.
\end{lemma}

\begin{proof}[Proof of Theorem~\ref{thm:homology-cob-inv}, Part II]
  Suppose $W$ is a homology cobordism between $M$ and~$M'$.  Recall
  that we have constructed bijections
  \begin{gather*}
    P_k(R,\pi_1(M)) \approx P_k(R,\pi_1(W)) \approx P_k(R,\pi_1(M')),\\
    \Hom(\pi_1(M),H) \approx \Hom(\pi_1(W),H) \approx \Hom(\pi_1(M'),H).
  \end{gather*}
  Fix $\alpha \in P_k(R,\pi_1(M))$ and $\psi\colon \pi_1(M)\to H$, and
  as an abuse of notation, we denote by $\alpha$ and $\psi$
  the corresponding representations and homomorphisms of $\pi_1(M')$ and
  $\pi_1(W)$ as well.

  For convenience write $Q=Q(H)$.  Since $H_*(W,M)=0$, we have
  $H_*(W,M;Q^k)=0$ by Lemma~\ref{lem:homology-cohn-local-coeff}, and
  so $H_*(M;Q^k)\cong H_*(W;Q^k)$.  Similarly for $M'$ and $W$, and
  from this the conclusion on the ranks follows immediately.

  Now pick a self--dual basis $\BB$ for $H_*(M;Q^k)$ and pick a
  self--dual basis $\BB'$ for $H_*(M';Q^k)$. Recall that $H_*(\partial
  M;Q^k)=H_*(\partial M';Q^k)=0$ since $\psi$ is admissible (see
  Section \ref{section:twi3mfd}).  In particular $\BB$ gives rise to a
  basis for $H_*(M,\partial M;Q^k)$ which we also denote by $\BB$.
  Note that $\BB$ and $\BB'$ give rise to a self--dual basis for
  $H_*(\partial W;Q^k)$ which we denote by $\BB\oplus \BB'$.

  Since $H_*(W,M;Q^k)=0$, we have $\Lambda(W,\phi)=1$.  Therefore, by
  Theorem~\ref{thm:torsion-duality}, we have
  $\tau^{\alpha\otimes\psi}(\partial W)\doteq
  \tau^{\alpha\otimes\psi}(\partial W,\BB\oplus \BB')\doteq 1$.
  Applying Theorem~\ref{thm:alg-prop-torsion} to the excision short
  exact sequence
  \[
  0 \to C_*(M';Q^k) \to C_*(\partial W;Q^k) \to C_*(M,\partial M;Q^k)
  \to 0
  \]
  and then applying Lemma~\ref{lem:torsion-duality} we obtain
  \[
  \begin{aligned}
    1 = \tau^{\alpha\otimes\psi}(\partial W,\BB\oplus \BB') &
    =\tau^{\alpha\otimes\psi}(M',\BB')
    \cdot \tau^{\alpha\otimes\psi}(M,\partial M,\BB) \cdot \tau_0 \\
    &= \tau^{\alpha\otimes\psi}(M',\BB') \cdot
    \overline{\tau^{\alpha\otimes\psi}(M,\BB)} \cdot \tau_0,
  \end{aligned}
  \tag{$*$}\]
  where $\tau_0$ is the torsion of the homology long exact sequence
  \[
  \to H_i(M';Q^k) \to H_i(\partial W;Q^k) \to H_i(M,\partial
  M;Q^k) \to H_{i-1}(M';Q^k) \to
  \]
  which is based by $\BB', \BB'\oplus \BB$ and~$\BB$.
  It follows easily that $\tau_0=1$.  Therefore, multiplying
  $(*)$ by $\tau^{\alpha\otimes\psi}(M,\BB)$, the conclusion follows.
\end{proof}

%==================================================================
\section{Link concordance and twisted torsion}\label{section:lc}

It is well known that if two links are concordant, then their
exteriors are homology cobordant. More precisely, if $m$-component
links $L$ and $L'$ in $S^3$ are concordant, then the exterior $W$ of a
concordance has boundary $E_L \cup m(S^1\times S^1 \times[0,1]) \cup
-E_{L'}$ where we glue the $i$-th meridian (longitude) of $L$ to the
$i$-th meridian (longitude) of $L'$ (Recall that links are assumed to
be \emph{ordered} collections of circles).  Rounding corners and
extending collars outward, we can assume that $\partial W=E_L \cup
-E_{L'}$ and $\partial E_L=E_L\cap E_{L'}=\partial E_{L'}$.  Using
Alexander duality it is straightforward to verify that $W$ is indeed a
homology cobordism.

Therefore, we can apply the invariance of torsion
(Theorem~\ref{thm:homology-cob-inv}) immediately.  Furthermore, in
case of links, we can make further observations on the correspondence
between representations.  The main aim of this section is to provide a
precise description of this correspondence.

First, recall that given an oriented link $L\subset S^3$ with $m$
components, we can naturally identify $H_1(E_L)$ with $\Z^m$ in such a
way that the $i$th positive meridian $\mu_i$ of $L$ represents the
$i$th standard basis of~$\Z^m$.  Therefore $\Hom(\pi_1(E_L),H)$ is
identified with $\Hom(\Z^m,H)$, where $H$ denotes a nontrivial free
abelian group as in the previous sections.  Suppose two links $L$ and
$L'$ are concordant, in particular the concordance exterior $W$ is a
homology cobordism between $E_L$ and $E_{L'}$.  Recall that in
Theorem~\ref{thm:homology-cob-inv}, we defined a bijection $\Psi\colon
\Hom(\pi_1(E_L),H) \to \Hom(\pi_1(E_{L'}),H)$ using $H_1(E_L)\cong
H_1(W)\cong H_1(E_{L'})$.  Since the meridians of $L$ and $L'$ are
freely homotopic in $W$, it is easily seen that the bijection $\Psi$
induces the identity on $\Hom(\Z^m,H)$, that is, a homomorphisms
$\pi_1(E_L) \to H$ and its image under $\Psi$ are identified with the
same homomorphism $\Z^m\to H$.

Recall that a homomorphism $\psi\colon \Z^m \to H$ is
\emph{admissible} if $\psi$ is an epimorphism which is nontrivial on
each $\mu_i$.  In this case the corresponding map $\pi_1(E_L) \to H$,
which will also be denoted by $\phi$, is admissible in the sense of
Section~\ref{sec:twisted invariants}.  Thus the twisted torsion
$\tau^{\alpha\otimes\phi}(E_L)$ can be defined for any unitary representation
$\alpha\colon \pi_1(E_L) \to \gl(k,R)$.

The correspondence between representations of $\pi=\pi_1(-)$ for
concordant links is better described in terms of representations of
the lower central quotients~$\pi/\pi_q$.  It causes no loss of
generality, since any representation of $\pi$ factoring through a
$p$-group factors through $\pi/\pi_q$ for some $q$ since $p$--groups are nilpotent.  To proceed we will need some
technicalities discussed in the following subsection.

%===================================================================
\subsection{Concordance and lower central series}

In this subsection we state some known results and some folklore
results on link concordance and lower central series.
% , and Milnor's
% $\bar\mu$-invariants.

Let  $L\subset S^3$ be an ordered, oriented $m$--component link.
We denote the meridians by $\mu_1,\dots,\mu_m$ and by abuse of notation we will denote the 
corresponding elements in $\pi_1(E_L)$ by $\mu_1,\dots,\mu_m$ as well.
 We now denote by $F$ the free group on $m$ generators $x_1,\dots,x_m$.
Following ideas of Levine (see \cite[p.~101]{Le94}) we define an \emph{$F/F_q$-structure} for $L$ to be 
  a homomorphism $\varphi\colon \pi\to F/F_q$ such
    that for any $i=1,\dots,m$ the element $\varphi(\mu_i)$ is a
    conjugate of $x_i$.
(Here recall that given a group $\pi$ we denote by $\pi_q$ the $q$th lower central subgroup.) 
Furthermore, we refer to a link $L$
equipped with an $F/F_q$--structure as an
\emph{$F/F_q$--link}.

 A~\emph{special automorphism of $F/F_q$} is an
automorphism $h$ of $F/F_q$ such that $h(x_i)$ is a conjugate of $x_i$
for each~$i$.

\begin{lemma}\label{lemmapsi}
Let $L$ be an ordered oriented $m$-component link. We write $\pi=\pi_1(E_L)$.
\bn
\item Any $F/F_q$-structure $\varphi:\pi\to F/F_q$ induces an isomorphism $\pi/\pi_q\to F/F_q$.
\item   If $\varphi$ and $\varphi'$ are $F/F_q$--structures for $L$, then there exists a special automorphism $\Theta$ such that
  $\varphi'=\Theta\circ \varphi$.
\en
\end{lemma}

\begin{proof}
We denote by  $\gamma\colon F \to \pi$ the map 
which sends  $x_i$ to $\mu_i$.
We also pick a map $\psi\colon F\to F$ such that the map $F\xrightarrow{\psi}F\to F/F_q$ agrees with $\varphi\circ \gamma$.
Finally note that $\gamma$ descends to a map $F/F_q\to \pi/\pi_q$ which we denote by $\ol{\gamma}$.
We now have the following commutative diagram, where the vertical maps are the obvious projection maps:
% \[
%  \begin{diagram}
%    \node{\End_{\Z[\pi]}(\Z[\pi]^n)} \arrow{e,t}{\alpha\otimes\psi}
%    \arrow{s,l}{f} \arrow{se,t}{\alpha}
%    \node{\End_R(R[H]^{nk})} \arrow{e,t}{\det} \arrow{s,r}{\epsilon_R}
%    \node{R[H]} \arrow{s,r}{\epsilon_R}
%    \\
%    \node{\End_{\Q[P]}(\Q[P]^n)} \arrow{e}
%    \node{\End_R(R^{nk})} \arrow{e,t}{\det}
%    \node{R}
%  \end{diagram}
%  \]
 \[
 \begin{diagram}
 \node{F}\arrow[2]{e,t}{\psi} \arrow{s}\arrow{se,t}{\gamma} \node{}\node{F}\arrow{s} \\
 \node{F/F_q}\arrow{e,t}{\ol{\gamma}} \node{\pi/\pi_q}\arrow{e,t}{\varphi}\node{F/F_q.}\end{diagram}
 \]
%\[
%\xymatrix{ F\ar[rr]^\psi\ar[d]\ar[dr]^\gamma && F\ar[d] \\ F/F_q\ar[r]^{\ol{\gamma}}& \pi/\pi_q\ar[r]^\varphi
%& F/F_q.}\]
It follows from Stallings' theorem~\cite{St65} applied to $\psi$ and from the commutativity of the diagram that the map $\varphi\circ \ol{\gamma}$ is an isomorphism. On the other hand it follows from Milnor's theorem \cite[Theorem~4]{Mi57}
that $\ol{\gamma}$ is surjective. We now conclude that $\varphi$ is an isomorphism. This concludes the proof of (1).

The second statement is an immediate consequence of (1). 
\end{proof}

Note that it follows from Lemma \ref{lemmapsi} (1) and \cite{Mi57} that a link admits an $F/F_q$-structure if and only if Milnor's $\bar\mu$-invariants of the form $\bar\mu(i_1\cdots
    i_{q-1})$ defined in \cite{Mi57} vanish for~$L$.

%==============================================
\subsection{Obstructions to links being concordant}

Now we are ready to derive the following theorem:

\begin{theorem}\label{thm:mainthm}\label{mainthm}
  Suppose two $m$--component ordered links $L$ and $L'$ are
  concordant, $\psi\colon\Z^m\to H$ is an admissible homomorphism, and
  $R$ is a subring of $\C$ closed under complex conjugation.  If
  $\varphi$ and $\varphi'$ are arbitrary $F/F_q$--structures for $L$
  and $L'$, respectively, then there exists a special automorphism
  $\Theta$ of $F/F_q$ such that for any $\alpha \in P_k(R,F/F_q)$ we
  have
  \[
  \rank(L,\psi,\alpha\circ \varphi)=\rank(L',\psi,\alpha \circ
  \Theta\circ \varphi')
  \]
  and
  \[
  \tau^{ (\alpha\circ \varphi)\otimes\psi}(L)=\pm dh \cdot
  \tau^{(\alpha \circ\Theta \circ\varphi')\otimes \psi}(L') \in
  Q(H)^\times/ N(Q(H))
  \]
  for some $d\in \det(\a(F/F_q))$ and $h\in H$.
\end{theorem}

\begin{proof}
 Let $W$ be the exterior of a concordance, $E=E_L$, and
  $E'=E_{L'}$.  Since $H_*(E) \cong H_*(W) \cong H_*(E')$ we can apply
  Stallings' theorem \cite{St65} to conclude that the inclusion maps
  induce isomorphisms
  \[
  \pi_1(E)/\pi_1(E)_q \xrightarrow{\cong} \pi_1(W)/\pi_1(W)_q
  \xleftarrow{\cong} \pi_1(E')/\pi_1(E')_q.
  \]
  Observe that the composition sends meridians of $L$ to (conjugates
  of) meridians of~$L'$.  It follows that the composition
  \[
  \pi_1(E') \to \pi_1(E')/\pi_1(E')_q \cong \pi_1(W)/\pi_1(W)_q \cong
  \pi_1(E)/\pi_1(E)_q \to F/F_q
  \]
  gives an $F/F_q$--structure on~$L'$ which we denote by $\varphi''$. 

  For $\alpha\in
  P_k(R,F/F_q)$, $\alpha\circ\varphi \in P_k(R,\pi_1(E))$ corresponds to
  $\alpha\circ\varphi'' \in P_k(R,\pi_1(E'))$ under the bijection
  $\Phi$ in Theorem~\ref{thm:homology-cob-inv}, by the definition of
  $F/F_q$--concordance and the definition of~$\Phi$ (see
  Subsection~\ref{subsec:correspondence-of-representations}).  Now by
  Theorem~\ref{thm:homology-cob-inv}, we have
  \begin{gather*}
    \rank(L,\psi,\alpha\circ \varphi)=\rank(L',\psi,\alpha\circ
    \varphi'')
    \\
    \tau^{ (\alpha \circ \varphi)\otimes\psi}(L)= \tau^{(\alpha\circ
      \varphi'')\otimes \psi}(L').
  \end{gather*}

  By Lemma~\ref{lemmapsi} there exists a special automorphism $\Theta$ such that
  $\varphi''=\Theta\circ \varphi'$.  The theorem now immediately follow   from these observations.
\end{proof}

The following corollary is equivalent to Theorem \ref{mainthmintro}
given in the introduction.

\begin{corollary}\label{cor:maincor}
  Suppose $L$ is an $m$--component slice link, $\psi\colon \Z^m \to H$
  is an admissible homomorphism, and $R$ is a subring of $\C$ closed
  under complex conjugation.  Then for any $\alpha \in P_k(R,\pi(L))$,
  we have $ \rank(L,\psi,\a)=k(m-1)$ and
  \[
  \tau^{\a\otimes \psi}(L)= \pm dh\cdot \prod_{i=1}^m
  \det(\id-\alpha(\mu_i)t_i)^{-1}
   \in
  Q(H)^\times/ N(Q(H))
  \]
  for some $d\in \det(\a(\pi(L)))$ and $h\in H$.
\end{corollary}

\begin{proof}
  Write $\pi=\pi_1(E_L)$. Let $\a\in P_k(R,\pi(L))$.  By definition
  $\a$ factors through a $p$-group. Since $p$-groups are nilpotent it
  follows that there exists a $q$ such that $\a$ factors through
  $\pi/\pi_q$.  Since $L$ is slice, by the proof of
  Theorem~\ref{thm:mainthm} there exists an $F/F_q$--structure
  $\varphi\colon \pi\to \pi/\pi_q\xrightarrow{\cong} F/F_q$ on~$L$. In
  particular there exists $\alpha_0 \in P_k(R,F/F_q)$ such that
  $\alpha=\alpha_0\circ \varphi$.  Since $L$ is concordant to the
  trivial link, we may by Theorem \ref{mainthm} assume without loss of
  generality that $L$ is already the trivial link.  (Precisely
  speaking, this requires the change of $\alpha_0$ to $\alpha_0 \circ
  \Theta$ for some special automorphism $\Theta$ of $F/F_q$, but this
  can be ignored since
  $\det(\id-\alpha(\mu_i)t_i)=\det(\id-\alpha_0(x_i)t_i)$ is left
  unchanged when $x_i$ is replaced with its conjugate.)  The corollary
  now follows from the explicit calculation for the unlink, which will
  be done in Theorem~\ref{cor:unlink}.
\end{proof}

%===========================================================
\section{Representation theory of $p$--groups}\label{section:prep}

In this section we collect a few basic facts of the theory of representations factoring through $p$-groups. This summary will be useful in the later discussion of examples.

Given $k$ we denote by $P(k)\subset \gl(\C,k)$ the subgroup of
permutation matrices, i.e. matrices with exactly one non--trivial
entry in each row and in each column, and all the non--trivial entries
are equal to one.  Let $R\subset \C$ be a subring which is closed
under complex conjugation. We then write $D(R,k)\subset \gl(R,k)$ for
the subgroup of diagonal matrices.

\begin{proposition} \label{proppequalspd} \label{lem:ppowerrep} Let
  $\pi$ be a group with generators $g_1,\dots,g_n$.  Let $\a \colon
  \pi\to \gl(\C,k)$ be a representation.  We write $X_i=\a(g_i),
  i=1,\dots,n$. Let $p$ be a prime. Then $\a$ factors through a
  $p$--group if and only if there exists a matrix $Q\in \gl(\C,k)$
  such that for $i=1,\dots,n$ we can write $QX_iQ^{-1}=P_iD_i$ with
  $P_i\in P(k)$ and $ D_i\in D(\C,k)$ such that
  \begin{enumerate}
  \item $P_1,\dots,P_n$ generate a subgroup of $p$--power order, and
  \item the diagonal entries of $D_1,\dots,D_n$ are $p$--power roots
    of unity.
  \end{enumerate}
\end{proposition}

\begin{proof}
  We will say that a matrix $D$ is $p$--diagonal if $D$ is diagonal
  and if all the diagonal entries are $p$--power roots of unity.  We
  will several times make use of the following two basic observations:
  \begin{enumerate}
  \item[(i)] If $P,P'\in P(k)$ and $D,D'\in D(R,k)$ with $PD=P'D'$,
    then $P=P'$ and $D=D'$.
  \item[(ii)] Let $P,P'\in P(k)$ and $D,D'\in D(R,k)$, then
    $PDP'D'=PP'\cdot P'^{-1}DP'D'$ and $P'^{-1}DP'D'$ is a diagonal
    matrix, furthermore, if $D,D'$ are $p$--diagonal, then
    $P'^{-1}DP'D'$ is also $p$--diagonal.
  \end{enumerate}
  Now assume that we have a representation $\a \colon \pi\to
  \gl(\C,k)$ with the following property: there exists a matrix $Q\in
  \gl(\C,k)$ such that for $i=1,\dots,n$ we can write
  $Q\a(g_i)Q^{-1}=P_iD_i$ with $P_i\in P(k), D_i\in D(\C,k)$ which
  satisfy conditions (1) and~(2).

  We claim that $\a(\pi)$ is a $p$--group.  Without loss of generality
  we can assume that $Q$ is the identity. Let $Y$ be any element in
  the group $\a(\pi)$, we can write
  $Y=(P_{i_1}D_{i_1})^{\epsilon_1}\cdot \dots \cdot
  (P_{i_r}D_{i_r})^{\epsilon_r}$ for some $i_j\in \{1,\dots,n\},
  \epsilon_j\in \{\pm 1\}$.  We have to show that $Y$ has $p$--power
  order.  Let $X=P_{i_1}^{\epsilon_1}\cdot \dots \cdot
  P_{i_r}^{\epsilon_r}$.  By assumption $X^{p^s}=\id$ for some
  $s$. Using the above observations we can write
  \[
  Y^{p^s} =\big((P_{i_1}D_{i_1})^{\epsilon_1}\cdot \dots \cdot
  (P_{i_r}D_{i_r})^{\epsilon_r}\big)^{p^s} =X^{p^s} \cdot D
  \]
  for some $p$--diagonal matrix~$D$. It follows that $Y^{p^s}=D$ (and
  hence $Y$) has $p$--power order.  This shows that $\a(\pi)$ is a
  $p$--group.

  Now let $\a \colon \pi\to \gl(\C,k)$ be a representation which
  factors through a $p$--group $P$.  We write $X_i=\a(g_i),
  i=1,\dots,n$.  By \cite[Corollary~4.9]{We03} any representation
  $\a\colon P\to \gl(\C,k)$ of the $p$-group $P$ is induced from a
  representation of degree 1. This means that there exists a finite
  index subgroup $Q\subset P$ and a one--dimensional representation
  $Q\to \gl(\C,1)$ such that $\a$ is given by the natural $P$--left
  action on $\C[ P]\otimes_{\C[ Q]} \C$. Note that the
  one--dimensional representation is necessarily of $p$--power order.
  Now pick representatives $p_1,\dots,p_l$ for $P/Q$. These
  representatives defines a basis for $\C[ P]\otimes_{\C[ Q]} \C$, and
  writing $\a$ with respect to this basis we see that $\a$ is of the
  required type.
\end{proof}

\section{Boundary links}\label{section:boundarylinks}

Boundary links form an important subclass of links which has been
intensely studied over the years.  In this section we will show how to
calculate twisted invariants for boundary links and we will study the
twisted invariants of boundary links that are `boundary slice'.

\subsection{Invariants of boundary links}

\begin{definition}
  A \emph{boundary link} is an $m$--component oriented link in $S^3$
  which has $m$ disjoint Seifert surfaces, i.e. there exist $m$
  disjoint oriented surfaces $V_1,\dots,V_m \subset S^{3}$ such that
  $\partial(V_i)=L_i, i=1,\dots,m$. We refer to such a disjoint
  collection of Seifert surfaces as a \emph{boundary link Seifert
    surface}.
\end{definition}

\begin{remark}
  Not every link is a boundary link. For example it is clear that a
  boundary link has trivial linking numbers. Furthermore, Milnor's
  \cite{Mi57} $\bar{\mu}$--invariants are trivial for a boundary link,
  and in fact trivial for any link concordant to a boundary link.
  Cochran and Orr \cite{CO90,CO93} showed that there also exist links
  with vanishing $\bar{\mu}$--invariants which are not concordant to a
  boundary link (cf.\ also \cite{GL92} and \cite{Le94}).
\end{remark}

Throughout this section let $F$ be the free group on the generators
$x_1,\dots,x_m$. An \emph{$F$--link} is a pair $(L,\v)$ where $L$ is
an oriented $m$--component link in $S^{3}$ and $\v\colon \pi(L) \to F$
is an epimorphism sending an $i^{th}$ meridian to $x_i$. Note that
$\v$ gives rise to a map from $E_L$ to the wedge of $m$ circles, and
by a standard transversality argument such a map then gives rise to a
boundary link Seifert surface.  Conversely, the existence of a
boundary link Seifert surface $V$ for $L$ produces such an epimorphism
$\pi(L)\to F$ by the Thom-Pontryagin construction.

Given an oriented boundary link $L$ with disjoint oriented surfaces
$V_1,\dots,V_m \subset S^{3}$ we can define the `boundary link Seifert
matrix' as follows: Let $g_i$ be the genus of $V_i$. For $i=1,\dots,m$
pick a basis $v_{i,1},\dots,v_{i,2g_i}$ for $H_1(V_i)$. For $k,l\in
\{1,\dots,m\}$ we define $A_{kl}$ to be the $2g_k\times 2g_l$--matrix
given by $(A_{ij})_{kl}=\lk(v_{i,k},(v_{j,l})_+)$ where $(v_{j,l})_+$
denotes the positive push--off of $v_{j,l}$ from $V_j$ into $S^3\sm
V$.  We now view $(A_{ij})_{kl}$ as an $m\times m$--matrix of matrices
and we refer to it as a \emph{boundary link Seifert matrix}.  We refer
to \cite{Lia77} and \cite{Ko85,Ko87} for details and for more
information on boundary link Seifert matrices.

We now turn to the computation of twisted invariants of boundary
links.  Let $(L,\varphi)$ be an $F$-link and $A=(A_{ij})$ a
corresponding boundary link Seifert matrix where $A_{ij}$ is an
$r_i\times r_j$--matrix. Let $r=\sum_{i=1}^mr_i$ and
\[
T:=\diag(\underbrace{t_1,\dots,t_1}_{r_1},
\dots,\underbrace{t_m,\dots,t_m}_{r_m}).
\]
We view $A^t-AT$ as an $r\times r$--matrix over $R[t_1^{\pm
  1},\dots,t_m^{\pm 1}]$. Let $\psi\colon \Z^m \to H$ be an admissible
epimorphism to a nontrivial free abelian group $H$ and let $\a\colon
F\to \gl(R,k)$ be a representation.  Note that all entries of the
matrix $(\a\otimes \psi)(A^t-AT)$ are sums of monomials in one
variable, we can therefore unambiguously define the matrix $(\a\otimes
\psi)(A^t-AT)$ to be the $rk\times rk$--matrix over $R[H]$ given by
replacing each summand $at_i^j$ of an entry of $A^t-AT$ by the matrix
$a\a(x_i)^j\psi(t_i)\in \gl(R[H],k)$. (Here and throughout this
section we will identify the additive group $\Z^m$ with the
multiplicative group generated by $t_1,\dots,t_m$.)

\begin{theorem}\label{thm:compboundary}\label{thm:tauboundary}
  Let $(L,\varphi)$ be an $F$--link. Let $A$ be a boundary link
  Seifert matrix corresponding to Seifert surfaces given by $\varphi$.
  Let $\psi\colon \pi(L)\to H$ be an admissible epimorphism to a
  nontrivial free abelian group $H$.  Let $R\subset \C$ be a subring
  closed under complex conjugation, and let $\a\colon \pi(L)\to
  \gl(R,k)$ be a unitary representation which factors through the
  epimorphism~$\varphi$.  Then the following hold:
  \begin{enumerate}
  \item We have $\rank(L,\psi,\a)\geq k(m-1)$ and equality holds if
    and only if $\det((\a\otimes \psi)(A^t-AT))\ne 0$,
  \item If $\rank(L,\psi,\a)= k(m-1)$, then
    \[
    \tau^{\a\otimes \psi}(L)\doteq \det((\a\otimes \psi)(A^t-AT))\cdot
    \prod\limits_{i=1}^m \det(\id-\a(x_i)t_i)^{-1} \in
    \frac{Q(H)^\times}{N(Q(H))}.
    \]
  \item If $\a$ factors through a $p$--group, then $\rank(L,\psi,\a)=
    k(m-1)$.
  \end{enumerate}
\end{theorem}

\begin{proof}
  Let $V_1\cup \dots \cup V_m\subset E_L$ be a properly embedded
  boundary link Seifert surface for $L$ corresponding to $\varphi$
  together with bases $v_{i,1},\dots,v_{i,2g_i}$ for $H_1(V_i)$ such
  that $(A_{ij})$ is the corresponding Seifert matrix, i.e.
  $(A_{ij})_{kl}=\lk(v_{i,k},(v_{j,l})_+)$.  Let $C=S^3\sm (\nu
  V_1\cup\dots \cup \nu V_m)$. Note that the induced maps
  $\pi_1(V_i)\to F, i=1,\dots,m$ and $\pi_1(C)\to F$ are trivial, in
  particular $\a$ and $\psi$ restricted to $\pi_1(V_i), i=1,\dots,m$
  and $\pi_1(C)$ are trivial.  We therefore get the following exact
  Mayer--Vietoris sequence:
%  \[ \ba{ccccccccccccc} &\bigoplus\limits_{i=1}^m H_1(V_i)\otimes_\Z
%  R[H]^k&\xrightarrow{}&
%  H_1(C)\otimes_\Z R[H]^k&\to& H_1(E_L;R[H]^k)&\to&\\
%  \to &\bigoplus\limits_{i=1}^m H_0(V_i)\otimes_\Z
%  R[H]^k&\xrightarrow{}& H_0(C)\otimes_\Z R[H]^k&\to&
%  H_0(E_L;R[H]^k)&\to&0.\ea
%  \]
  \begin{multline*}
    \bigoplus\limits_{i=1}^m H_1(V_i)\otimes_\Z R[H]^k \to
    H_1(C)\otimes_\Z R[H]^k \to H_1(E_L;R[H]^k) \\
    \to \bigoplus\limits_{i=1}^m H_0(V_i)\otimes_\Z R[H]^k \to
    H_0(C)\otimes_\Z R[H]^k \to H_0(E_L;R[H]^k) \to 0
  \end{multline*}
  (We refer to \cite[Section~3]{FK06} for details.)  Now let
  $c_{i,1},\dots,c_{i,2g_i}, i=1,\dots,m$ be the basis for $H_1(C)$
  dual to the basis $v_{i,1},\dots,v_{i,2g_i}, i=1,\dots,m$,
  i.e. $\lk(c_{i,j},v_{k,l})=\delta_{ik}\cdot \delta_{jl}$.  Using
  these bases it is well--known that the map
  \[
  H_1(V_i)\otimes_\Z R[H]^k\xrightarrow{}
  H_1(C)\otimes_\Z R[H]^k
  \]
  is represented by $(\a\otimes \psi)(A^t-AT)$.  We write
  $n=\sum_{i=1}^m 2g_ik$.  Note that $(\a\otimes \psi)(A^t-AT)$ is an
  $n\times n$-matrix.  We now have
%  \[ \ba{rcl}&&\rank(L,\psi,\a)\\
%  &=&\dim_{Q(H)} H_1(E_L;Q(H)^k)\\
%  &=&\dim \left(\bigoplus\limits_{i=1}^m H_0(V_i)\otimes_\Z Q(H)^k\right)-\dim\%left( H_0(C)\otimes_\Z Q(H)^k\right)+\dim \left(H_0(E_L;Q(H)^k\right)\\
%  &&+\dim_{Q(H)}(H_1(C)\otimes Q(H)^k)-\mbox{rank}((\a\otimes \psi)(A^t-AT))\\
%  &=&km-k+0+n-\mbox{rank}((\a\otimes \psi)(A^t-AT)).\ea
%  \]
  \begin{align*}
    \rank(L,\psi,\a) &= \dim_{Q(H)} H_1(E_L;Q(H)^k)\\
    &=\dim \left(\bigoplus\limits_{i=1}^m H_0(V_i)\otimes_\Z Q(H)^k\right)-\dim\left( H_0(C)\otimes_\Z Q(H)^k\right)\\
    &\qquad+\dim \left(H_0(E_L;Q(H)^k\right)+\dim_{Q(H)}(H_1(C)\otimes Q(H)^k)\\
    &\qquad-\mbox{rank}((\a\otimes \psi)(A^t-AT))\\
    &=km-k+0+n-\mbox{rank}((\a\otimes \psi)(A^t-AT)).
  \end{align*}
  In particular $\rank(L,\psi,\a)\geq km-k$, and equality holds if and
  only if the rank of the $n\times n$-matrix $(\a\otimes
  \psi)(A^t-AT)$ is equal to $n$.  This proves (1).

  We now turn to the calculation of $\tau^{\a\otimes \psi}(L)$.
  Suppose that $\rank(L,\psi,\a)= k(m-1)$.  For $i=1,\dots,m$ we pick
  a point $u_i\in V_i$.  Note that $\{u_i,
  \{v_{i,1},\dots,v_{i,{2g_i}}\}\}$ defines a basis for the free
  $\Z$--module $H_*(V_i)$ and
  \[
  \tau(V_i;\{u_i, \{v^i_1,\dots,v^i_{2g_i}\}\})=\pm 1 \in \Z
  \]
  since $\pm 1$ are the only units in the ring $\Z$.  We now denote by
  $e_1,\dots,e_k$ the standard basis of $Q^k$. For $i=1,\dots,m$ we
  then let
  \begin{align*}
    \VV_{i0}&=\{u_i\otimes e_j\}_{j=1,\dots,k},\\
    \VV_{i1}&=\{v_{i,j}\otimes e_l\}_{j=1,\dots,2g_i, l=1,\dots,k}.
  \end{align*}
  Note that for $i=1,\dots,m$ the set $\{\VV_{i0},\VV_{i1}\}$ defines
  a basis for the free $Q(H)$--module $H_*(V_i)\otimes_{\Z}Q(H)^k$ and
  \[
  \tau(V_i;\{\VV_{i0},\VV_{i1}\})=\tau(V_i;\{u_i,
  \{v_{i,1},\dots,v_{i,{2g_i}}\}\})^k=\pm 1 \in \Z.
  \]
  We now also pick a point $b\in C$ and for $i=1,\dots,m-1$ we denote
  by $S_i$ the result of gluing $V_i\times -1$ and $V_i\times 1$
  together along $\partial V_i\times [-1,1]$.  Note that
  $S_1,\dots,S_{m-1}$ are a basis for $H_2(C)$.  We now let
  \begin{align*}
    \CC_{0}&=\{p\otimes e_j\}_{j=1,\dots,k},\\
    \CC_{1}&=\{c_{i,j}\otimes e_l\}_{i=1,\dots,m, j=1,\dots,2g_i, l=1,\dots,k},\\
    \CC_2&=\{S_i\otimes e_j\}_{i=1,\dots,m-1, j=1,\dots,k}.
  \end{align*}
  Note that the set $\{\CC_{0},\CC_{1},\CC_2\}$ defines a basis for
  the free $Q(H)$--module $H_*(C)\otimes_{\Z}Q(H)^k$ and a similar
  argument as above shows that
  \[
  \tau(C;\{\CC_{0},\CC_{1},\CC_2\})=\pm 1 \in \Z.
  \]
  Now note that $\rank(L,\psi,\a)=k(m-1)$ implies by (1) that the map
  \[
  \bigoplus\limits_{i=1}^m H_1(V_i)\otimes_\Z Q(H)^k\to
  H_1(C)\otimes_\Z Q(H)^k
  \]
  is injective. This implies that the map
  \[
  H_2(C)\otimes Q(H)^k\to H_2(E_L;Q(H)^k)
  \]
  is an isomorphism, and we denote by $\BB_2$ be the image of $\CC_2$
  under this isomorphism.  By a slight abuse of notation we also write
  $\CC_2:=\{S_i\otimes e_j\}_{i=1,\dots,m-1, j=1,\dots,k}$.

  For $i=1,\dots,m$ we now write $P_i=\id-\a(x_i)t_i$.  Furthermore
  for $i=1,\dots,m$ we let $\mu_i$ be a meridian of $L_i$ which is
  based at $b$ and which intersects $V_i$ geometrically once in the
  positive direction and which is disjoint from $V_j, j\ne i$.

  We now denote by $p\colon \widetilde{E_L}\to E_L$ the $F$--cover
  corresponding to~$\varphi$.  We pick a point $\ti{b}$ over $b$ and
  for $i=1,\dots,m$ we denote by $\ti{\mu}_i$ the component of
  $p^{-1}(\mu_i)$ such that $\partial \ti{\mu}_i=\ti{b}-x_i\cdot
  \ti{b}$.

  \begin{claim}
    For $i\in \{1,\dots, m-1\}$ and $j\in \{1,\dots,k\}$ we define
    \[
    b_{ij}':= \ti{\mu}_i\otimes P_i^{-1}e_j - \ti{\mu}_{i+1}\otimes
    P_{i+1}^{-1}e_j.
    \]
    Then the following hold:
    \begin{enumerate}
    \item $\{b_{ij}'\}$ is a basis for $H_1(E_L;Q(H)^k)$,
    \item for any $i,j,k$ and $l$ we have
      \[
      b_{ij}'\cdot ( S_{r}\otimes e_s)=
      \begin{cases}
        \delta_{js},& \mbox{if }r=i, \\
        -\delta_{js},&\mbox{if }r=i+1.
      \end{cases}
      \]
    \item for any $i$ and $j$ we have
      \[
      r(b_{ij}')=u_i\otimes P_i^{-1}e_j -u_{i+1}\otimes P_{i+1}^{-1}e_j.
      \]
    \end{enumerate}
  \end{claim}

  In order to prove the claim, first note that
  \begin{align*}
    \partial (\ti{\mu}_i\otimes e_j)&=(1-x_i)\ti{b}\otimes e_j\\
    &=\ti{b}\otimes (\a\otimes \psi)(1-x_i)e_j\\
    &= \ti{b} \otimes  (\id-\a(x_i)t_i)e_j=\ti{b}\otimes P_ie_j.
  \end{align*}
  It therefore follows that
  \[
  \partial(b_{ij}')=\ti{b}\otimes P_iP_{i}^{-1}e_j-\ti{b}\otimes
  P_{i+1}P_{i+1}^{-1}e_j=0.
  \]
  The calculation of $b_{ij}'\cdot ( S_{r}\otimes e_s)$ follows easily
  from the definitions.  This shows in particular that the $b_{ij}'$
  are linearly independent and hence form a basis for
  $H_1(E_L;Q(H)^k)$ Finally it is straightforward to verify that
  \[
  r(b_{ij}')=u_i\otimes P_i^{-1}e_j -u_{i+1}\otimes P_{i+1}^{-1}e_j.
  \]
  This concludes the proof of the claim.

  For $i=1,\dots,m-1$ and $j=1,\dots,k$ we now let
  \[
  b_{ij}=\sum_{r=i}^{m-1} b_{rj}'.
  \]
  We write $\CC_1:=\{b_{ij}\}$.  It follows immediately from the above
  claim that $\CC_1$ is a basis, and that $\CC_1$ is dual to $\CC_2$.
  We now consider the following short exact sequence of
  $Q(H)$--complexes:
  \[
  0\to \bigoplus\limits_{i=1}^m C_*(V_i;Q(H)^k)\to C_*(C;Q(H)^k)\to
  C_*(E_L;Q(H)^k)\to 0
  \]
  together with the above bases. It follows from Theorem
  \ref{thm:mi66} (4) together with the above calculations and from the
  definition of torsion that
  \[
  \tau^{\a\otimes \psi}(L,\{\BB_1,\BB_2\})= \det((\a\otimes
  \psi)(A^t-AT))\cdot \tau(C),
  \]
  where $C$ is the following complex with the canonical bases:
%\[ 0\to Q(H)^{k(m-1)} \xrightarrow{\bp P_1^{-1} &0&\dots &0    \\ -P_2^{-1}&P_2%^{-1} &\dots&0 \\
%0&\ddots&\ddots &\vdots \\
%%0&\dots  & -P_{m-1}^{-1} &P_{m-2}^{-1}&0 \\
%0&\dots 0 & -P_{m-1}^{-1} &P_{m-1}^{-1} \\
% 0&\dots  & 0 &-P_{m}^{-1}\ep} Q(H)^{km}\xrightarrow{\bp P_1 &\dots &P_m\ep}Q(H%)^k\to 0.\]
%
  \begin{gather*}
    0\to Q(H)^{k(m-1)} \xrightarrow{f} Q(H)^{km}\xrightarrow{g}Q(H)^k\to 0
    \\
    \begin{aligned}
      f&=\bp P_1^{-1} &0&\dots &0    \\ -P_2^{-1}&P_2^{-1} &\dots&0 \\
      0&\ddots&\ddots &\vdots \\
      % 0&\dots  & -P_{m-1}^{-1} &P_{m-2}^{-1}&0 \\
      0&\dots & -P_{m-1}^{-1} &P_{m-1}^{-1} \\
      0&\dots  & 0 &-P_{m}^{-1}\ep
      \\
      g&=\bp P_1 &\dots &P_m\ep
    \end{aligned}
  \end{gather*}
  It now follows (cf.\ e.g.\ \cite[Theorem~2.2]{Tu01}) that
  \[
  \tau(C)=\prod\limits_{i=1}^m \det(P_i)^{-1}.
  \]
  This concludes the proof of the second statement of the theorem.

  Finally note that $\det(A-A^t)=\pm 1$ (see \cite{Ko87}), in
  particular if $\a$ factors through a $p$--group, then it follows
  immediately from Theorem \ref{thm:levine} that $\det((\a\otimes
  \psi)(A^t-AT))\ne 0$.  This concludes the proof of part (3) of the
  theorem.
\end{proof}

The following corollary follows immediately from taking $A$ to be the
trivial matrix.

\begin{corollary}\label{cor:unlink}
  Let $L\subset S^3$ be the $m$--component unlink in $S^3$ with
  meridians $\mu_1,\dots,\mu_m$.  Let $\psi\colon \Z^m\to H$ be an
  admissible homomorphism.  Let $R\subset \C$ be a subring closed
  under complex conjugation, and let $\a\colon \pi(L)\to \gl(R,k)$ be
  a unitary representation. Then
  \[
  \rank(L,\psi,\a)=k(m-1)
  \]
  and
  \[
  \tau^{\a\otimes \psi}(L)=\pm dh \cdot \prod\limits_{i=1}^m
  \det\big(\id-\psi(\mu_i)\a(\mu_i)\big)^{-1} \, \in
  Q(H)^\times/N(Q(H))
  \]
  with $d\in \det(\a(\pi(L)))$ and $ h\in H$.
\end{corollary}

\subsection{Boundary slice links}

Let $L\subset S^3$ be an $m$--component boundary link.  We say $L$ is
\emph{boundary slice} if there exist disjointly embedded 3--manifolds
$W_1,\dots,W_m\in D^4$ such that for $i=1,\dots,m$, the boundary
$\partial W_i$ is the union of a Seifert surface and a slice disk
for~$L_i$.  It is known that a boundary slice link is boundary slice
with respect to \emph{any} boundary link Seifert surface.  We note
that boundary slice links are slice; it is a long-standing open
question whether the converse holds for boundary links.

% \begin{remark}
%   It is known that if $L$ is boundary slice, then for any boundary
%   link Seifert surfaces $V_1,\ldots,V_m$ for $L$, there exist
%   $W_1,\ldots,W_m$ with $\partial W_i=V_i \cup ($a slice disk for
%   $L_i)$ (e.g., see~\cite{Ko87}).  In other words, a boundary slice
%   link is boundary slice with respect to \emph{any} boundary link
%   Seifert surface
% \end{remark}

%\begin{remark}

Obstructions to being boundary slice has been studied by several
authors, including Cappell and Shaneson \cite{CS80}, Duval
\cite{Du86}, Ko \cite{Ko87}, Mio \cite{Mi87},
Sheiham~\cite{Sh03,Sh06}.  The following torsion obstruction to being
boundary slice is a consequence of these works.

% It is very difficult to find counter-examples since
% one can easily see that any ribbon boundary link is boundary slice.

%\end{remark}

% \begin{remark}
%   \begin{enumerate}
%   \item The notion of boundary slice generalizes to the notion of
%     boundary link concordance of pairs $(L,V)$ where $L$ is a boundary
%     link and $V$ a boundary link Seifert surface (see \cite{Ko87} for
%     details).
%   \item The notions of boundary slice and boundary link concordance
%     can be generalized to higher dimensions and Ko
%     \cite[Proposition~2.11]{Ko87} has shown that in higher dimensions
%     boundary link concordance classes of pairs $(L\subset S^{n+2},V)$
%     naturally give rise to a group $C_n(B_m)$.
%   \item Cappell and Shaneson \cite{CS80} have shown that any even
%     dimensional boundary link is boundary slice, i.e. $C_{2k}(B_m)=0$
%     for all $k$.  Furthermore Cappell and Shaneson \cite{CS80}, Duval
%     \cite{Du86}, Ko \cite{Ko87} and Mio \cite{Mi87} gave algebraic
%     descriptions of the group $C_{2k-1}(B_m)$ for $k \geq 2$. This
%     group was finally shown by Sheiham \cite{Sh03,Sh06} to be
%     isomorphic to $\Z^\infty\oplus \Z_2^\infty \oplus \Z_4^\infty
%     \oplus \Z_8^\infty$.
%   \end{enumerate}
% \end{remark}

\begin{theorem}\label{thm:boundaryslice}
  Let $(L,\varphi)$ be an $m$-component $F$-link which is boundary
  slice. Let $\a\colon \pi(L)\xrightarrow{\varphi} F\to \gl(R,k)$ be a
  unitary representation where $R\subset \C$ is a subring closed under
  complex conjugation. Let $\psi\colon \Z^m \to H$ be an admissible
  homomorphism. If $ \rank(L,\psi,\a)=k(m-1)$, then
  \[
  \tau^{\a\otimes \psi}(L) =\pm dh\cdot \prod\limits_{i=1}^m
  \det\big(\id-\psi(\mu_i)\a(\mu_i)\big)^{-1} \, \in Q(H)^\times/
  N(Q(H))
  \]
  for some $d\in \det(\a(\pi(L)))$ and $h\in H$.
\end{theorem}

\begin{proof}
  Let $V=V_1\cup \dots \cup V_m$ be a Seifert surface for $L$
  corresponding to $\varphi$ and let $(A_{ij})$ be the corresponding
  boundary link Seifert matrix.  We denote by $g_i$ the genus of
  $V_i$.  By the first remark of this subsection, there are
    disjointly embedded 3--manifolds $W_i$ in $D^4$ such that
    $\partial W_i = V_i \cup ($a slice disk for $L_i)$.  By Ko
    \cite[Lemma~3.3]{Ko87}, it follows that $(A_{ij})$ is metabolic.
  This means that for $i=1,\dots,m$ there exists an invertible
  $2g_i\times 2g_i$ matrix $P_i$ such that each $P_iA_{ij}P_j^t$ is of
  the form
  \[
  \bp 0 & C \\ D&E \ep
  \]
  where $0$ is a $g_i\times g_j$-matrix.  (In fact Ko \cite{Ko87} has
  shown that in higher odd dimensions $A$ being metabolic is in fact a
  sufficient condition for being boundary slice.)  It follows easily
  that $(\a\otimes \psi)(A^t-AT)$ is equivalent to a matrix of the
  form
  \[
  \pm h \bp 0& X \\ \ol{X}^t &Y\ep
  \]
  where $h\in H$ and $X$ and $Y$ are matrices over $R[H]$ of size
  $k\sum_i g_i$.  It now follows that
  \[
  \det ((\a\otimes \psi)(A^t-AT))=\pm u\det(R)\cdot \det(\ol{R}^t)=\pm
  u \det(R)\cdot \ol{\det(R)}
  \]
  for some $u\in H$.  The theorem now follows immediately from Theorem
  \ref{thm:tauboundary} and the assumption that $
  \rank(L,\psi,\a)=k(m-1)$.
\end{proof}

\begin{remark}
  Let $L$ be a boundary link. According to Theorem \ref{mainthm} the
  twisted torsion corresponding to any representation factoring
  through a $p$--group gives a sliceness obstruction.  On the other
  hand, according to Theorems \ref{thm:compboundary} and
  \ref{thm:boundaryslice} the twisted torsions corresponding to many
  more representations give an obstruction to being boundary
  slice. This is very similar to the situation in \cite{Fr05} where
  the vanishing of the eta invariant for (boundary) slice links is
  investigated, as well as~\cite{CKo99,Sm89} (cf.\
  \cite[Theorem~4.8]{Fr05} and \cite[Theorem~4.3]{Fr05}). Levine
  \cite{Le07} used work of Vogel \cite{Vo90} to resolve the apparent
  discrepancy in \cite{Fr05} between the vanishing of the eta
  invariants for boundary links which are slice and those which are
  boundary slice.
\end{remark}

%===========================================================
\section{Computation for satellite links}\label{section:satellitebing}

%===========================================================
\subsection{The satellite construction}

Let $L\subset S^3$ be an $m$--component oriented link and let
$K\subset S^3 $ be a knot. Let $A\subset E_L$ be a simple closed
curve, unknotted in $S^3$. Then $S^3\sm \nu A$ is a solid torus. Let
$\phi\colon \partial(\overline{\nu A})\to
\partial(\overline{\nu K})$ be a diffeomorphism which sends a meridian
of $A$ to a longitude of $K$, and a longitude of $A$ to a meridian of
$K$. The space
\[
\left({S^3\sm \nu A}\right) \cup_{\phi} \left({S^3\sm \nu K}\right)
\]
is diffeomorphic to $S^3$. The image of $L$ is denoted by
$S=S(L,K,A)$. We say $S$ is the \emph{satellite link with companion
  $K$, orbit $L$ and axis $A$}. Put differently, $S$ is the result of
replacing a tubular neighborhood of $K$ by a link in a solid torus,
namely by $L\subset {S^3\sm \nu A}$. Note that $S$ inherits an
orientation from $L$.

An important example is given by letting $L$ be the unlink with two
components and $A\subset E_L$ as in Figure \ref{fig:bd}.
\begin{figure}[h]
  \begin{center}
    \includegraphics[scale=0.95]{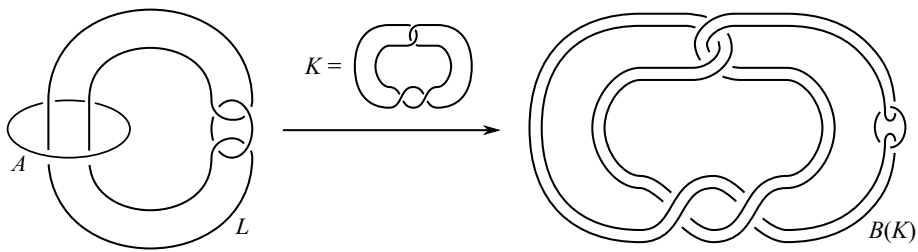}
    \caption{The Bing double of the Figure 8 knot.}
    \label{fig:bd}
  \end{center}
\end{figure}
The corresponding satellite knot $S(L,K,A)$ is called the \emph{Bing
  double of $K$} and referred to as $B(K)$.

We now return to the discussion of satellite links in general.  Note
that the abelianization map $\pi_1(S^3\sm \nu K)\to \Z$ gives rise to
a map of degree one $S^3\sm \nu K$ to $\overline{\nu A}$ which is a
diffeomorphism on the boundary.  In particular we get an induced map
\[
E_S = \left({S^3\sm \nu A \sm \nu L}\right) \cup_{\phi} \left({S^3\sm
    \nu K}\right) \to \left({S^3\sm \nu A \sm \nu L}\right) \cup
\overline{\nu A}= E_L
\]
which we denote by $f$. Note that $f$ is a diffeomorphism on the
boundary and that $f$ induces an isomorphism of homology groups.

Before we state the next lemma note that the curve $A$ defines an
element in $A\in \pi(L)$ which is well--defined up to conjugation.

\begin{lemma} \label{lemmahofsat} Let $L,K,A, S=S(L,K,A)$ and $f$ as
  above.  Let $Q$ be a subfield of $\C$ and let $\a\colon \pi(L)\to
  \gl(Q,k)$ be a unitary representation. We denote the representation
  $\pi(S)\xrightarrow{f}\pi(L)\xrightarrow{\a} \gl(Q,k)$ by $\a$ as
  well.  Let $\psi\colon \Z^m\to H$ be an admissible homomorphism such
  that $\psi(A)=0$.  Denote by $z_1,\dots,z_k$ the eigenvalues of
  $\a(A)$ and let $\Delta_K\in \zt$ be a fixed representative of the
  Alexander polynomial of $K$.  Then the following hold:
  \begin{enumerate}
  \item $\rank(S,\psi,\a)=\rank(L,\psi,\a)$ if and only if
    $\Delta_K(z_i)\ne 0$ for $i=1,\dots,k$,
  \item if $\Delta_K(z_i)\ne 0$ for $i=1,\dots,k$, then
    \[
    \tau^{\a\otimes \psi}(S)= \tau^{\a\otimes \psi}(L)\cdot
    \prod_{i=1}^k \Delta_K(z_i) \in Q(H)^\times/N(Q(H))
    \]
    up to multiplication by an element of the form $\pm dh$ with $
    d\in \det(\a(\pi(L)))$ and $ h\in H$.
  \end{enumerate}
\end{lemma}

Before we give the proof of Lemma \ref{lemmahofsat} we first state and
prove the following result which is an immediate consequence of Lemma
\ref{lemmahofsat} and Theorem \ref{mainthmintro}.

\begin{corollary} \label{corhofsat} Let $L$ be an oriented
  $m$--component slice link and let $K,A$ and $S=S(L,K,A)$ as above.
  Let $Q$ be a subfield of $\C$ closed under complex conjugation which
  has the unique factorization property and let $\a\colon \pi(L)\to
  \gl(R,k)$ be a representation which factors through a $p$--group.
  Let $\psi\colon \Z^m\to H$ be an admissible homomorphism.  Denote by
  $z_1,\dots,z_k$ the eigenvalues of $\a(A)$ and let $\Delta_K\in \zt$
  be a fixed representative of the Alexander polynomial of $K$.
  Assume that $\psi(A)=0$.  If $S=S(L,K,A)$ is slice, then
  $\Delta_K(z_i)\ne 0$ for $i=1,\dots,k$ and
  \[
  \prod_{i=1}^k \Delta_K(z_i)=\pm d\cdot q\ol{q} \in Q
  \]
  for some $d\in \det(\a(\pi(L)))$ and $q\in Q$.
\end{corollary}

\begin{proof}[Proof of Lemma \ref{lemmahofsat}]
  We write $T=\ol{\nu A}$. Consider the following commutative diagram
  of short exact sequences of cellular chain complexes (where we write
  $V=Q(H)^k$):
%\[ \ba{cccccccccc}
%0 \to \hspace{-0.2cm}& C_*(\partial(T);V) & \hspace{-0.2cm}\to
%\hspace{-0.2cm}&C_*({E_L\sm \mbox{int} \, T};V) &\hspace{-0.2cm}\oplus \hspace{%-0.2cm} & C_*({S^3\sm \nu
%C};V)& \hspace{-0.2cm}\to \hspace{-0.2cm}&
%C_*(E_S;V)&\hspace{-0.2cm}\to\hspace{-0.2cm}& 0 \\[2mm]
%&\downarrow  id &&\downarrow id &&\downarrow f&&\downarrow f&\\[2mm]
%0 \to \hspace{-0.2cm}& C_*(\partial(T);V) &\hspace{-0.2cm}\to\hspace{-0.2cm}
%&C_*({E_L\sm \mbox{int} \, T};V) &\hspace{-0.2cm}\oplus\hspace{-0.2cm}& C_*(T;V%)
%&\hspace{-0.2cm}\to \hspace{-0.2cm}& C_*(E_L;V)&\hspace{-0.2cm}\to\hspace{-0.2c%m}& 0.\ea
%\]
  \[
  \begin{diagram}\dgHORIZPAD=5pt \dgARROWLENGTH=1em
    \node{0}\arrow{e}
    \node{C_*(\partial(T);V)}\arrow{e}\arrow{s,r}{\id}
    \node{\hspace*{3em}}
    \node{C_*({E_L\sm \nu A};V) \oplus C_*({S^3\sm \nu C};V)}
    \arrow{s,r}{\id\oplus f}
    \node{\hspace*{3em}}\arrow{e}
    \node{C_*(E_S;V)}\arrow{e}\arrow{s,r}{f}
    \node{0}
    \\
    \node{0}\arrow{e}
    \node{C_*(\partial(T);V)}\arrow{e}
    \node{\hspace*{.5em}}
    \node{C_*({E_L\sm \nu A};V) \oplus C_*(T;V)}
    \node{\hspace*{.5em}}\arrow{e}
    \node{C_*(E_L;V)}\arrow{e}
    \node{0}
  \end{diagram}
  \]
  We assumed that $\psi(A)=0$. It follows that the map $\psi$
  restricted to $T$ and restricted to $S^3\sm \nu K$ is trivial. In
  particular $C_*^{\a\otimes \psi}(E_K;Q(H)^k)=C_*^\a(E_K;Q^k)\otimes
  Q(H)$ and $C_*^{\a\otimes \psi}(T;Q(H)^k)=C_*^\a(T;Q^k)\otimes
  Q(H)$.  It is a consequence of \cite[Section~5]{Go78} that
  $f_*\colon H_1(E_K;Q^k)\to H_1(T;Q^k)$ is always surjective and it
  is an isomorphism if and only if $\Delta_K(z_i)\ne 0$ for
  $i=1,\dots,k$. The first statement of the lemma now follows
  immediately from the commutative diagram of long exact homology
  sequences corresponding to the above diagram.

  Now suppose that $\Delta_K(z_i)\ne 0$ for $i=1\,\dots,k$.

  \begin{claim}
    The map $f_*\colon H_j(E_K;Q^k)\to H_j(T;Q^k)$ is an isomorphism
    for any~$j$.
  \end{claim}

  We already saw above that $f_*\colon H_1(E_K;Q^k)\to H_1(T;Q^k)$ is
  an isomorphism.  Note that $E_K$ is homotopy equivalent to a
  2--complex and that $T$ is homotopy equivalent to 1--complex.  In
  particular $H_2(T;Q^k)=0$ and it follows from the long exact
  sequence that $H_2(E_K;Q^k)=0$ as well.  Note that $\pi_1(E_K)\to
  \pi_1(T)$ is surjective, it follows that $f_*\colon H_0(E_K;Q^k)\to
  H_0(T;Q^k)$ is surjective as well.  On the other hand we have that
  $\chi(E_K)=0=\chi(T)$, hence
  \[
  b_0(E_K;Q^k)=b_1(E_K;Q^k)+k\chi(E_K)=b_1(T;Q^k)+k\chi(T)=b_0(T;Q^k),
  \]
  i.e. $f_*\colon H_0(E_K;Q^k)\to H_0(T;Q^k)$ is in fact an
  isomorphism.  This concludes the proof of the claim.

  Let $\DD_i$ be $Q$--bases for $H_i(T;Q^k), i=0,1$.  We endow
  $H_i(E_K;Q^k)$ with the corresponding basis $f_*^{-1}(\DD_i),
  i=0,1$.  It follows from the long exact sequence corresponding to
  $f\colon E_K\to T$ that $H_*(f\colon E_K\to T;Q^k)=0$.  It follows
  from Theorem \ref{thm:mi66} (4) that
  \[
  \tau^\a(T;\{\DD_i\}_{i=0,1})=\tau^\a(E_K,\{f_*^{-1}(\DD_i)\}_{i=0,1})\cdot
  \tau^\a(f\colon E_K\to T).
  \]
  (Here and in the following lines all equalities of torsion are up to
  multiplication by an element of the form $\pm dh$ with $d\in
  \det(\a(\pi(L)))$ and $h\in H$.)  Note that the $\DD_i$ are also
  $Q$--bases for $H_i(T;Q(H)^k)=H_i(T;Q^k)\otimes Q(H), i=0,1$. In
  particular the calculations in the previous lines also work for the
  torsions corresponding to twisting by $\a\otimes \psi$.

  Let $B$ be a Seifert matrix for the knot $K$. It follows easily from
  \cite[Section~5]{Go78} and \cite[Theorem~2.2]{Tu01} that
  \[
  \tau^\a(f\colon E_K\to T)^{-1}=\prod_{i=1}^k
  \det(B^t-Bz_i)=\prod_{i=1}^k \Delta_K(z_i).
  \]
  Now pick a $Q(H)$--basis $\BB$ for $H_1(E_L;V)$ and denote the dual
  basis for $H_2(E_L;V)$ by $\BB'$.  Finally pick bases $\EE_i,
  i=0,1,2$ for $H_i(\partial T;Q(H)^k)$.  Note that the bases
  $\DD_i,\BB,\EE_i$ give rise to bases $\FF_i, i=0,1,2$ for
  $H_*(E_L\sm \mbox{int} \, T;Q(H)^k)$.  It now follows from Theorem
  \ref{thm:mi66} (4) that
  \begin{align*}
    \tau^{\a\otimes \psi}(E_S,\{f^{-1}(\BB_i)\})&= \tau^{\a\otimes
      \psi}(E_L\sm \mbox{int} \, T;\{\FF_i\})\cdot \tau^{\a\otimes
      \psi}(E_K;\{f^{-1}(\DD_i)\}) \\
    & \qquad \cdot \tau^{\a\otimes \psi}(\partial
    T,\{\EE_i\})^{-1}
    \\
    \tau^{\a\otimes \psi}(E_L,\{\BB_i\})&= \tau^{\a\otimes \psi}(E_L\sm
    \mbox{int} \, T;\{\FF_i\})\cdot \tau^{\a\otimes
      \psi}(T;\{\DD_i\})\cdot \tau^{\a\otimes \psi}(\partial
    T,\{\EE_i\})^{-1}
  \end{align*}
  (Here and in the next line the equalities of torsion are up to
  multiplication by an element of the form $\pm dh$ with $d\in
  \det(\a(\pi(L)))$ and $h\in H$.)  Combining these results we now
  obtain that
  \[
  \tau^{\a\otimes \psi}(E_S,\{f^{-1}(\BB_i)\})=\tau^{\a\otimes
    \psi}(E_L,\{\BB_i\})\cdot \prod_{i=1}^k \Delta_K(z_i).
  \]
  This implies the desired equality.
\end{proof}

\subsection{Bing doubles}

We consider the Bing double $B(K)$ of a knot $K$, in order to provide
an example for a detailed computation.  Note that if $K$ is a slice
knot, then $B(K)$ is in fact a slice link (cf.\
\cite[Section~1]{Ci06}).  The converse is a well--known folklore
conjecture.

Given any knot $K$ all the `classical' sliceness obstructions (e.g.\
multivariable Alexander polynomial, Levine--Tristram signatures) of
its Bing double are trivial (cf.\ \cite[Section~1]{Ci06} for a
beautiful survey).  It is therefore an interesting question to
determine which methods and invariants detect the non-sliceness of
Bing doubles.  For interesting results, we refer to
Cimasoni~\cite{Ci06}, Harvey~\cite{Ha08}, Cha~\cite{Ch07,Ch09}, and
Cha-Livingston-Ruberman~\cite{CLR08}.  In particular in \cite{CLR08}
it was proved that if $K$ is not algebraically slice, $B(K)$ is not
slice.  Also it is known that there are algebraically slice knots with
non-slice Bing doubles~\cite{CK08,CHL08}.

%  The first result
% was obtained by Cimasoni~\cite{Ci06} who used Sheiham's work
% \cite{Sh03} to show that if $B(K)$ is algebraically boundary slice,
% then $K$ is algebraically slice.  Harvey used certain von-Neumann
% $\rho$-invariants from metabelian covers to prove that $B(K)$ is not
% slice if the integral of the Levine--Tristram signature of $K$ over
% the unit circle vanishes~\cite{Ha08}.  The first author used his
% $L$-group valued invariants obtained from iterated $p$-covers to prove
% that the Levine--Tristram signature function of $K$ is nontrivial,
% then $B(K)$ is not slice~\cite{Ch07}.  He also obtained obstructions
% to $B(K)$ being slice from discriminant invariants.  These results
% were strengthened by the first author, Livingston and Ruberman
% \cite{CLR08} who proved the following: if $K$ is not algebraically
% slice, then $B(K)$ is not slice.

Now let $K=4_1$ be the Figure 8 knot. It is well--known that $K$ is
not slice.  For example this follows from \cite{FM66} since the
Alexander polynomial $\Delta_{4_1}(t)=t^{-1}-3+t$ is not a norm,
namely not of the form $\pm t^k f(t)f(t^{-1})$.
% On the other hand $K\# K$ is slice, i.e. $K$ is torsion in the knot
% concordance group.  , in particular all Levine--Tristram signatures
% of $K$ vanish.  Furthermore the signature invariants of
% \cite{CKo99}, \cite{Fr05} and \cite{Ha08} vanish.  We will now show
% that our invariants can detect that the Bing double of the Figure 8
% knot is not slice.
In what follows we show that an appropriate twisted torsion invariant
of the Bing double $B(4_1)$ is not a norm, and consequently $B(4_1)$
is not slice.  (The fact that $B(4_1)$ is not slice had first been
shown in \cite{Ch07} using the discriminant invariant.)

Let $x,y$ be the meridians of the two component unlink $L$ and denote
by $A$ the unknotted curve in the definition of Bing doubles.  It is
well-known that $A$ represents the commutator $[x,y]$. We write
$\xi=e^{2\pi i/8}$.  Note that $\Z[\xi]$ is a UFD (cf.\
\cite[Theorem~11.1]{Wa97}).  Consider the representation $\a\colon
\pi(L)=\ll x,y\rr\to \gl(\Z[\xi],8)$ given by
\[ \a(x)=\bp
0&0&0&\xi   &0&0&0&0 \\
1&0&0&0   &0&0&0&0 \\
0&1&0&0   &0&0&0&0 \\
0&0&-i&0   &0&0&0&0 \\
0&0&0&0   &0&0&0&\xi \\
0&0&0&0   &-1&0&0&0 \\
0&0&0&0   &0&1&0&0 \\
0&0&0&0   &0&0&-\xi&0 \ep
\]
and
\[
\a(y)=\bp
0&0&0&0   &\xi&0&0&0 \\
0&0&0&0   &0&0&0&-i \\
0&0&0&0   &0&0&\xi^3&0 \\
0&0&0&0   &0&\xi^7&0&0 \\
i&0&0&0   &0&0&0&0 \\
0&0&0&\xi   &0&0&0&0 \\
0&0&\xi^3&0   &0&0&0&0 \\
0&-1&0&0   &0&0&0&0 \ep.
\]
It follows easily from Proposition \ref{lem:ppowerrep} that $\a$
factors through a 2--group.  We denote the representation
$\pi_1(E_{B(K)})\xrightarrow{f} \pi(L)\to \gl(\Z[\xi],8)$ by $\a$ as
well, and we let $\psi$ be the isomorphism $H_1(E_{B(K)})\cong
H_1(E_L)\to \Z^2$.  (Here $K=4_1$.)  We compute the eigenvalues of
$\a(A)=\a([x,y])$ to be
\[ \{\pm 1,\pm i, \pm e^{2\pi i/16}, \pm e^{2\pi 3i/16}\}.\]
We then calculate
\begin{multline*}
  \Delta_K(1)\cdot \Delta_K(-1) \cdot \Delta_K(i)\cdot \Delta_K(-i)\\
  \cdot \Delta_K(e^{2\pi i/16})\cdot \Delta_K(-e^{2\pi i/16})\cdot
  \Delta_K(e^{2\pi 3i/16})\cdot \Delta_K(-e^{2\pi 3i/16})
\end{multline*}
to equal~$2115$.  If $B(K)$ was slice, then by Corollary
\ref{corhofsat} we would have
\[
2115=\pm d q\cdot \ol{q}
\]
for some $d\in \det(\a(\pi(L)))$ and for some $q$ in the quotient
field of $\Z[\xi]$.  Note that $q\ol{q}$ is a positive real number and
note that $\det(\a(\pi(L)))\cap \R=\pm 1$.  It therefore follows that
if $B(K)$ is slice, then $2115=q\cdot \ol{q}$.  Since $\Z[\xi]$ is a
UFD we furthermore deduce that $q$ actually lies in $\Z[\xi]$.  Now
note that $2115=47\cdot 45=47\cdot (6+3i)(6-3i)$. Since $\Z[\xi]$ is a
UFD we only have to show that $47$ is not a norm. A direct calculation
shows that any norm in $\Z[\xi]$ is of the form
\[
a^2+b^2+c^2+d^2+\sqrt{2}(a(b-d)+c(b+d))
\]
for some $a,b,c,d\in \Z$.  One can now easily deduce that $47$ is not
a norm in $\Z[\xi]$. This concludes the proof of the claim.

Note that the process of Bing doubling can be iterated. We refer to
the work of the first author \cite{Ch07}, the first author and Kim
\cite{CK08}, Cochran, Harvey and Leidy \cite{CHL08} and van Cott
\cite{vCo09} for interesting results on iterated Bing doubling.

% We conclude this paper with a short list of questions and open
% problems:
% \begin{enumerate}
% \item Let $K$ be a knot and $BD(K)$ its Bing double. Let $\psi\colon
%   H_1(BD(K))\to \Z^2$ an isomorphism.  Assume that $\tau^{\a\otimes
%     \psi}(BD(K))$ is a norm for any representation $\a$ factoring
%   through a $p$--group. Does it follow that $\Delta_K$ is a norm?
% \item Find the relationship between the concordance obstructions of
%   \cite{Ch07} and the obstructions coming from twisted torsion
%   invariant.  We expect that our concordance obstructions are
%   equivalent to the discriminant of the $L$-group valued invariant
%   presented in \cite{Ch07}. In particular we expect that the
%   obstructions of this paper are strong enough to show that none of
%   the iterated Bing doubles of the Figure 8 knot are slice (see
%   \cite{Ch07}).
% \end{enumerate}

\parindent=0mm

\end{document}